\documentclass[12pt,english]{amsart}
\usepackage[a4paper]{geometry}
\usepackage{amssymb,amsmath}
\usepackage{amsthm}
\usepackage{enumerate}
\usepackage[utf8]{inputenc}
\usepackage[english]{babel}
\usepackage{xcolor}
\usepackage{url}
\usepackage{graphicx}
\usepackage{tikz}
\usetikzlibrary{matrix,arrows,shapes.geometric}

\usepackage{csquotes}
\usepackage[backend=biber,style=numeric,url=false,isbn=false,doi=false,giveninits=true,hyperref=true,backref=true,]{biblatex}
\usepackage[bookmarks=true,hyperindex,pdftex,colorlinks,citecolor=red, linkcolor=blue]{hyperref}

\usepackage{babel}

\numberwithin{equation}{section}

\usepackage{tikz-cd}
\theoremstyle{plain}
\newtheorem{theorem}{Theorem}[section]

\newtheorem{proposition}[theorem]{Proposition}

\newtheorem{corollary}[theorem]{Corollary}

\theoremstyle{definition}
\newtheorem{definition}[theorem]{Definition}
\newtheorem{remark}[theorem]{Remark}
\newtheorem{example}[theorem]{Example}
\newtheorem{examples}[theorem]{Examples}

\newtheorem{problem}[theorem]{Problem}

\newcommand{\N}{\mathbb{N}}
\newcommand{\R}{\mathbb{R}}

\def\NN{\mathbb N}
\def\IN{\hbox{{\rm I}\kern-.13em{\rm N}}}
\def\ZZ{\Bbb Z}
\def\RR{\mathbb{R}}

\def\IR{\hbox{{\rm I}\kern-.13em{\rm R}}}

\def\aa{\alpha}

\def\dd{\delta}

\def\aim{\text{targeting}{}}

\renewcommand{\bar}{\overline}
\renewcommand{\tilde}{\widetilde}

\newcommand{\spa}{\operatorname{span}}

\makeatletter
\renewcommand{\tocsection}[3]{%
	\indentlabel{\@ifnotempty{#2}{\bfseries\ignorespaces#1 #2\quad}}\bfseries#3}
\renewcommand{\tocsubsection}[3]{%
	\indentlabel{\@ifnotempty{#2}{\ignorespaces#1 #2\quad}}#3}

\newcommand\@dotsep{4.5}
\def\@tocline#1#2#3#4#5#6#7{\relax
	\ifnum #1>\c@tocdepth 
	\else
	\par \addpenalty\@secpenalty\addvspace{#2}%
	\begingroup \hyphenpenalty\@M
	\@ifempty{#4}{%
		\@tempdima\csname r@tocindent\number#1\endcsname\relax
	}{%
		\@tempdima#4\relax
	}%
	\parindent\z@ \leftskip#3\relax \advance\leftskip\@tempdima\relax
	\rightskip\@pnumwidth plus1em \parfillskip-\@pnumwidth
	#5\leavevmode\hskip-\@tempdima{#6}\nobreak
	\leaders\hbox{$\m@th\mkern \@dotsep mu\hbox{.}\mkern \@dotsep mu$}\hfill
	\nobreak
	\hbox to\@pnumwidth{\@tocpagenum{\ifnum#1=1\bfseries\fi#7}}\par
	\nobreak
	\endgroup
	\fi}
\AtBeginDocument{%
	\expandafter\renewcommand\csname r@tocindent0\endcsname{0pt}
}
\def\l@subsection{\@tocline{2}{0pt}{2.5pc}{5pc}{}}
\makeatother

\def\pare{\text{par}}
\def\chil{\text{Chi}}
\newcommand{\proj}{\operatorname{proj}}
\newcommand{\ro}[2]{\varrho_#1^#2}
\newcommand{\abs}[1]{\left\lvert#1\right\rvert}

\newcommand{\ntres}[1]{{\left\vert\kern-0.25ex\left\vert\kern-0.25ex\left\vert #1 
		\right\vert\kern-0.25ex\right\vert\kern-0.25ex\right\vert}} 

\newcommand{\ndos}[1]{{\left\vert\kern-0.25ex\left\vert\kern-0.25ex #1 
		\kern-0.25ex\right\vert\kern-0.25ex\right\vert}} 

\newcommand{\free}[1]{\mathcal{F}(#1)} 

\def\lip{ \text{Lip}}

\def\sp{\hbox{{\rm span}}}

\begin{document}

\title{Hypercyclicity and
Lipschitz-free operators}

\author{Christian Cobollo, Romuald Ernst, Quentin Menet and Alfred Peris}

\address{Christian Cobollo,
Institut Universitari de Matemàtica Pura i Aplicada, Universitat Politècnica de València, Camí Vera S/N, 46022 València, SPAIN}
\email{chcogo@upv.es}

\address{Romuald Ernst,
Université du Littoral Côte d’Opale,  Laboratoire de Mathématiques Pures et Appliquées Joseph Liouville,  62228 Calais, FRANCE}
\email{romuald.ernst@math.cnrs.fr}

\address{Quentin Menet, Département de Mathématique, Université de Mons, 20 Place du Parc, 7000 Mons, BELGIUM}
\email{quentin.menet@umons.ac.be}

\address{Alfred Peris,
Institut Universitari de Matemàtica Pura i Aplicada, Universitat Politècnica de València, Camí Vera S/N, 46022 València, SPAIN}
\email{aperis@upv.es}

\keywords{Hypercyclic operators, shifts on trees, Lipschitz-free spaces, Linear Dynamics, Dynamical Systems}

\subjclass[2020]{47A16; 46B20; 37B05}

\begin{abstract}

We study linearizations of dynamical systems and some of its topological properties. Special attention is paid to the case of Lipschitz-free operators, and they are shown to model the dynamics of very general linearizations. We provide a new criterion, called the targeting property, for the (weakly) mixing property in a linear dynamical system through a (possibly) non-linear restriction of it. We then apply this criterion  to  the  linearization $T_f$ of a Lipschitz map $f\colon M\to M$, where $M$ is a metric space and the  operator $T_f$ is defined on the corresponding Lipschitz-free space $\free{M}$---the so-called Lipschitz-free operators. We show that, under some natural assumptions on the distance considered in $M$, the operator $T_f$ is hypercyclic if and only if the map $f$ has the targeting property.

\end{abstract}

\maketitle

\tableofcontents

\section{Introduction}\label{intro}

Through this document we will contribute to the study of the relation between a non-linear dynamical system and a corresponding linear dynamical system, understood as its \textit{linearization}. The interest of doing this relies on the possibility for a given linear operator $T$ satisfying certain conditions to deduce a behaviour on the whole space from the specific local behaviour on a (possibly non-linear) restriction on the original (linear) dynamical system.

Regarding the relation between linear and non-linear dynamical systems, Feldman~\cite{Feldman01} showed in 2001 that there exists a universal operator $T$ on the Hilbert space such that, for any continuous function $f$ on a compact metric space, there is a $T$-invariant compact subset of the Hilbert space so that the restriction of $T$ to this invariant set is topologically conjugate to $f$. Probably, this was the first attempt to connect in a precise way linear and non-linear dynamics, although we should also mention the work of Protopopescu \cite{Prot90} on (global) Carleman linearizations. Feldman's approach provides an operator that captures every observable dynamical behaviour on compact metric spaces.

It would be interesting to have a more precise connection between a continuous function and a related linear operator. In \cite{MP15}, it was shown that a close connection between non-linear and linear dynamics can be established, as soon as the map $f\colon K\to K$ fixes the point $0$ and the linear span of $K$ is dense in the topological vector space $X$ where we then extend $f$ to an operator $T$. More precisely, topologically mixing properties and Devaney chaos were shown to be ``extendable'' under very general assumptions, and this has been illustrated with examples including the Protopopescu approach with Carleman linearization, Feldman universal operator, and dynamics of what we will refer to as \textit{Lipschitz-free operators}\footnote{This class of operators are also referred to as \textit{Lipschitz operators} by many other authors. However, it looks like the notation differs between different sources in the literature, as sometimes is used for the linearization $T_f$, but some times to a map $f\in \lip_0(X,Y)$ between Banach spaces $X$ and $Y$, 
 and even to denote some other kind of operators outside the context of Lipschitz-free spaces (see e.g. \cite{Rob89} and \cite{wang12}). For the sake of clarity, we opted to use the term \textit{Lipschitz-free operator}.}, that is, linearization operators $T_f\colon \free{M}\to \free{N}$ of non-linear Lipschitz mappings $f\colon M\to N$, defined over the corresponding Lipschitz-free spaces. The latter was further developed in \cite{ACP21} and has been the subject of several papers. For instance, the set of recurrent vectors of such operators has been investigated in \cite{Tap24}, and the first and the last author \cite{CP25} recently studied disjoint transitivity and its generalizations associated with Furstenberg families, with emphasis and examples on extensions of maps to Lipschitz-free spaces. 

Specifically, this paper aims to introduce new criteria for non-linear maps so that, when linearized, the associated operator is hypercyclic and even weakly mixing. Our main tool to produce such linearizations will consist in considering the Lipschitz-free operators $T_f$ of non-linear lipschitz maps $f$. With this, in particular, we will contribute to extend the knowledge on the existence of dense orbits for this kind of operators. 

In \cite{ACP21}, the authors introduced the Hypercyclicity Criterion for Lipschitz-free operators and showed that if a Lipschitz map $f$ satisfies this criterion then the associated operator $T_f$ is hypercyclic and even weakly mixing. However this criterion is too demanding. As remarked in \cite{ACP21}, if we consider the compact space $M=\{0\}\cup\{\frac{1}{n}:n\in \mathbb{N}\}$ endowed with the usual distance in $\mathbb{R}$ and the Lipschitz map $f$ on $M$ given by $f(0)=0$, $f(1)=1$ and $f(\frac{1}{n})=\frac{1}{n-1}$ for $n\ge 2$, then $f$ does not satisfy the Hypercyclicity Criterion for Lipschitz-free operators while $T_f$ is hypercyclic and even weakly mixing. 

In parallel to this study of the dynamics of Lipschitz-free operators, the dynamics of weighted shifts on trees has also begun to be explored \cite{Ma17,AA25, LP25, GePapreprint, AC26}, and a characterization of hypercyclicity for these operators has even been achieved in \cite{GEPa23}. We will demonstrate later that weighted shifts on trees are, in fact, a special case of Lipschitz-free operators, and the knowledge gained about these operators can thus be used to unveil new criteria for hypercyclicity within the broader context of Lipschitz-free operators.

This is how we introduce in this paper the targeting property that will allow us to characterize the hypercyclicity of Lipschitz-free operators under certain conditions on the considered metric space. This characterization will apply to Lipschitz-free operators corresponding to weighted shifts on trees, but also to all Lipschitz-free operators defined on $M=\{0\}\cup\{\frac{1}{n}:n\in \mathbb{N}\}$. In particular, this will allow us to recover the hypercyclicity of the operator $T_f$ mentioned earlier through our new criterion.

The document is organized as follows: In Section 2 we first introduce general linearizations of discrete dynamical systems and present the required preliminaries on Lipschitz-free spaces and Lipschitz-free operators $T_f$. Then, a general result is established showing that Lipschitz-free operators constitute a universal model among linearizations of Lipschitz maps on metric spaces. We finally relate this kind of operators with the recently studied weighted backward shifts on sequence spaces defined over countable directed trees (see \cite{GEPa23}, etc), showing in particular that any weighted backward shift operator $B_\lambda$ on the real space $\ell_1(V)$ can be identified with a $T_f$ operator over a certain $\free{M_V}$.

Section 3 is motivated by the Hypercyclicity Criterion for Lipschitz-free operators stated in \cite{ACP21}. We introduce a blow-up/collapse criterion which serves as a non-linear version of the linear blow-up/collapse criterion to achieve the weak mixing property. In contrast, we show that there exist examples of linearizations that are weakly mixing but do not satisfy those criteria.

In Section 4, based on the characterization of hypercyclic weighted shifts on trees obtained in  \cite{GEPa23}, we introduce what we call the \textbf{\aim{} property} around a fixed point, a strictly weaker condition than the previously defined blow-up/collapse property.  Assuming the {\aim{} property} we obtain weakly mixing linearizations. This condition also characterizes hypercyclicity and the weak  mixing property on $T_f$ operators over Lipschitz-free spaces admitting a Schauder basis formed by deltas, or more generally on $T_f$ operators based on a metric space $M$ where points far from zero are well separated, such as $\{0\}\cup\{\frac{1}{n}:n\in \mathbb{N}\}$.

\section{Linearizations, Lipschitz-free operators and shifts on trees}

Our main  interest is on Lipschitz-free linearizations, but in many cases there are natural linearizations which are relevant. The first point is to set the general ground that will allow us to precise the linearization considered in each case. 

\begin{definition}
    Given a continuous map $f\colon W\to W$ on a topological space $W$ with a fixed point $0_W$, a continuous injection $j:W\to X$ in a topological vector space $X$, and an operator $T:X\to X$, we say that $(X,T)$ is a \emph{linearization} of $(W,f)$ if $T\circ j=j\circ f$, $j(0_W)=0$, and $\sp (j(W))$ is dense in $X$. 
\end{definition}

Lipschitz-free operators constitute the main paradigm to illustrate general results on the linearization of non-linear mappings. 
A thorough systematic introduction to Lipschitz-free spaces can be found in sources such as \cite{Wea18}. Here, we provide a concise overview of some fundamental properties and tools that are essential for developing the current work.

Consider a metric space $(M, d)$. By selecting a distinguished  point $0_M \in M$ (commonly referred to simply as $0$ when no ambiguity arises), $\lip_0(M)$ is defined as the space of Lipschitz functions $f \colon M \to \mathbb{R}$ that vanish at $0$, which forms a real Banach space when equipped with the norm given by the Lipschitz constant $\|\cdot\|_{\text{Lip}}$, defined as
\begin{equation*}
  \|f\|_{\text{Lip}} := \sup_{x \neq y \in M} \frac{|f(x) - f(y)|}{d(x,y)}.
\end{equation*}

The selection of the distinguished point $0$ is arbitrary within the category of Banach spaces, as the resulting spaces obtained by choosing different distinguished points turn out to be isometrically isomorphic as real Banach spaces--although additional structures may depend on the choice of the distinguished point, see Subsection \ref{subsec:base_point}. 
It is also noteworthy that it suffices to consider complete metric spaces, since $\lip_0$ (and therefore, the Lipschitz-free space) remains isometrically isomorphic regardless of whether $M$ or its completion is considered.

The map $\delta \colon M \to \lip_0(M)^*$, which assigns to each $x \in M$ its corresponding evaluation functional $\delta_x$, is an isometry. Moreover, it is straightforward to observe that $\delta_x$ and $\delta_y$ for $x \neq y$ are linearly independent. It is worth noting that the closed linear span of $\{\delta_x : x \in M\}$ forms a subspace of $\lip_0(M)^*$, and indeed serves as a predual of $\lip_0(M)$, referred to as $\mathcal{F}(M)$, i.e.,
\[
\mathcal{F}(M) := \overline{\text{span}}\{\delta_x : x \in M\} \ \ (\subset \lip_0(M)^*).
\]

This space is known as the \textit{Lipschitz-free space} of $M$. Intuitively, one can think of $\mathcal{F}(M)$ as the Banach space formed by taking $M$ and endowing it with a linear structure, where distinct (non-zero) points in $M$ are linearly independent, and the norm preserves the original metric of $M$. Hence, $\|\delta_x\| = d(x,0)$, and more generally, $\|\delta_x - \delta_y\| = d(x,y)$.

Since the elements of $\lip_0(M)$ are the continuous linear functionals on $\mathcal{F}(M)$, their action is clear. For any $g \in \lip_0(M)$ and $\sum_{i=1}^n a_i \delta_{x_i} \in \mathcal{F}(M)$, we have
\[
\langle g, \sum_{i=1}^n a_i \delta_{x_i} \rangle = \sum_{i=1}^n a_i \langle g, \delta_{x_i} \rangle = \sum_{i=1}^n a_i g(x_i).
\]

For any two distinct points $x, y \in M$, the associated \textbf{molecule} (sometimes referred to as an \textit{elementary molecule}) is defined as
\[
m_{x,y} := \frac{\delta_x - \delta_y}{d(x,y)} \in \mathcal{F}(M).
\]

While $\mathcal{F}(M)$ can be fully described using \textit{Dirac deltas}, molecules offer certain advantages. Since $\delta_x = d(x,0) \cdot m_{x,0}$ for any $x \in M$, we can express the Lipschitz-free space as the closed span of molecules
\[\mathcal{F}(M) = \overline{\text{span}}\{m_{x,y} : x,y \in M\}.\]\

The significance of these molecules becomes apparent when one realizes that they not only have norm one but also constitute a $1$-norming subset of $\mathcal{F}(M)$ for the norm $\|\cdot\|_{\text{Lip}}$, as $\langle g, m_{x,y} \rangle = \frac{g(x) - g(y)}{d(x,y)}$. Molecules also play a key role in the extreme structure of $\mathcal{F}(M)$. For example, it has been recently proven that all extreme points in a Lipschitz-free space are indeed molecules, which has been one of the central open problems in the field \cite{AlPerSm24p}.

Now, consider two metric spaces $(M,d)$ and $(N,p)$, each with distinguished points $0_M$ and $0_N$, respectively. The space of Lipschitz maps $f\colon M \to N$ that satisfy $f(0_M) = 0_N$ is denoted by $\lip_0(M,N)$. The following result is fundamental for connecting non-linear dynamics with linear dynamics via Lipschitz-free spaces.

\begin{theorem}\label{thm:linearization}
    Let $(M,d)$ and $(N,p)$ be metric spaces with distinguished points $0_M$ and $0_N$, respectively, and let $f \in \lip_0(M,N)$. Then, there exists a unique bounded linear operator $T_f \colon \mathcal{F}(M) \to \mathcal{F}(N)$ such that $\|T_f\| = \|f\|_{\text{Lip}}$, and the following diagram commutes:

    \begin{center}
    \begin{tikzpicture}
      \matrix (m) [matrix of math nodes,row sep=3em,column sep=4em,minimum width=2em]
      {
         M & N \\
         \mathcal{F}(M) & \mathcal{F}(N) \\};
      \path[-stealth]
        (m-1-1) edge node [left] {$\delta_M$} (m-2-1)
                edge node [above] {$f$} (m-1-2)
        (m-2-1.east|-m-2-2) edge node [below] {$T_f$} (m-2-2)
        (m-1-2) edge node [right] {$\delta_N$} (m-2-2);
    \end{tikzpicture}
    \end{center}
\end{theorem}

Therefore, every $f \in \lip_0(M,M)$ induces an operator $T_f \in \mathcal{L}(\mathcal{F}(M))$, and this correspondence opens the possibility of examining which conditions on the original dynamical system $(M,f)$ ensure a particular dynamical property on $(\mathcal{F}(M), T_f)$.

\subsection{Why Lipschitz-free operators?}\label{subsec:why}

Given a Lipschitz map $f\colon M\to M$ on a metric space $M$ with a fixed point, as mentioned before, there might be several possible linearizations, and some of them resulting in an operator $T$ on a non-metrizable topological vector space $X$. A key dynamical property, like hypercyclicity, could be difficult to check for $T$ in case that we cannot apply a Baire argument directly on $X$. This makes the case of Lipschitz-free operators special: If $M$ is separable then $\mathcal{F}(M)$ is a separable Banach space in which we can apply Baire arguments. Together with a universality property of Lipschitz-free spaces we will be able to use the Lipschitz-free operators as a ``tool'' to obtain dynamical properties that usually need Baire arguments for linearizations defined on complete locally convex spaces which are not metrizable. 

To do this, we first recall that, given continuous maps $S:Y\to Y$ and $T\colon X\to X$ on topological spaces $X$ and $Y$, $T$ is called  \textit{quasi-conjugate} to $S$ if there exists a continuous map $\phi:Y\to X$ with dense range
such that $T\circ \phi = \phi\circ S$, that is, the diagram
\begin{center}
    \begin{tikzpicture}
      \matrix (m) [matrix of math nodes,row sep=3em,column sep=4em,minimum width=2em]
      {
         Y & Y \\
         X & X \\};
      \path[-stealth]
        (m-1-1) edge node [left] {$\phi$} (m-2-1)
                edge node [above] {$S$} (m-1-2)
        (m-2-1.east|-m-2-2) edge node [above] {$T$} (m-2-2)
        (m-1-2) edge node [right] {$\phi$} (m-2-2);
    \end{tikzpicture}
\end{center}
commutes. The dynamical properties that we are interested on (hypercyclicity, topological (weakly) mixing property, etc) are preserved under quasi-conjugacies. This is the case for sequential hypercyclicity too. We recall that an operator $T\colon X\to X$ on a topological vector space $X$ is \emph{sequentially hypercyclic} if there is an $x\in X$ such that, for each $y\in X$, there is a subsequence of the orbit of $x$ that converges to $y$.

\begin{definition}\label{def:sL}
    A map $h:M\to X$ from a metric space $(M,d)$ into a locally convex space $X$ is \emph{seminorm-Lipschitz} if, for every continuous seminorm $p$ on $X$, there is $c_p>0$ such that 
    $$
    p(h(x)-h(y))\leq c_p d(x,y), \ \ \ \forall x,y\in M.
    $$
\end{definition}

This property is usually easy to check in particular examples. A weaker (and even easier to check) property is the following. 

\begin{definition}\label{def:wL}
    A map $h:M\to X$ from a metric space $(M,d)$ into a locally convex space $X$ is \emph{weakly Lipschitz} if, for every continuous linear functional $\phi\in X'$, the map $\phi\circ h: M\to \R$ is Lipschitz.
\end{definition}

Obviously, the seminorm-Lipschitz property is stronger than the weak Lipschitz property. It turns out that, for barrelled spaces, the converse is true. We recall that $X$ is \emph{barrelled} if every barrel in $X$ (i.e.,  convex, balanced, absorbing, and closed set) is necessarily a neighbourhood of $0$. Easy examples of barrelled spaces are countable inductive limits of Fréchet spaces. We refer to \cite{BPC} for more details and examples of barrelled spaces. 

\begin{proposition}\label{prop:wL-sL}
    If $h:M\to X$ is {weakly Lipschitz} from a metric space $(M,d)$ into a locally convex space $X$ whose strong dual $X'_b$ is barrelled, then $h$ is seminorm-Lipschitz. 
\end{proposition}

\begin{proof}
     We consider the family of continuous linear functionals on $X'_b$
 $$
 \Psi_{x,y}(\phi):=d(x,y)^{-1}(\phi(h(x))-\phi(h(y))), \ \ x,y\in M, \ x\neq y, 
 $$
 which is pointwise bounded by assumption. The Uniform Boundedness Principle for barrelled spaces yields that this family is strongly bounded, that is, bounded on bounded sets of $X'_b$. Given a continuous seminorm $p$ on $X$, the set 
 $$
 B_p:=\{ \phi \in X' \ : \ |\phi(x)|\leq p(x) \ \ \forall x\in X\}
 $$
 is a bounded set in $X'_b$, thus there is $c_p>0$ such that $|\Psi_{x,y}(\phi)|\leq c_p$ if $x,y\in M$, $x\neq y$, and $\phi \in B_p$. Therefore, $p(h(x)-h(y))\leq c_pd(x,y)$ for all $x,y\in M$. 
\end{proof}

We are now in conditions to establish the following result, of special relevance for linearizations resulting in an operator on a non-metrizable space. 

\begin{theorem}\label{thm:whyLF}
Let $T\colon X\to X$ be a linearization on a complete locally convex space $X$ of a Lipschitz map $f\colon M\to M$ on a metric space $(M,d)$ ($f(0_M)=0_M$) such that the associated injection $j\colon M\to X$ is seminorm-Lipschitz. Then $T$ is quasi-conjugate to the induced Lipschitz-free operator $T_f$ on $\mathcal{F}(M)$.
\end{theorem}

\begin{proof}
    Since $X$ is a complete locally convex space, it can be written as a projective limit $X= \proj {\mathcal X}$ of a projective spectrum $\mathcal X=(X_\alpha)_{\alpha\in I}$ of Banach spaces, where $I$ is a directed index set, and the operators $\ro \beta \alpha \colon X_\beta \to X_\alpha$ for $\alpha \le \beta$, are the spectral maps ($\ro \beta \alpha \circ \ro \gamma \beta = \ro  \gamma \alpha$ and $\ro \alpha \alpha =I_{X_\alpha}$, the identity on $X_\alpha$, for any $\alpha\leq\beta\leq \gamma$). Let $\varrho^\alpha\colon X\to X_\alpha$ be the natural projection, so that $\ro   \alpha \beta \circ\varrho^\alpha=\varrho^\beta$ for $\beta \leq \alpha$. 

    The universality property of the Lipschitz-free Banach space $\mathcal{F}(M)$ yields that, for each $\alpha\in I$, the Lipschitz map $g_\alpha:=\varrho^\alpha\circ j\colon  M\to X_\alpha$ factorizes through $\mathcal{F}(M)$ as $g_\alpha=\mu^\alpha\circ \delta_M$ with an operator $\mu^\alpha\colon \mathcal{F}(M)\to X_\alpha$. By definition of projective limit, there exists an operator  $\phi\colon \mathcal{F}(M)\to X$ such that $\phi\circ \delta_M=j$. Finally, since $\overline{\spa(j(M))}=X$, we conclude that $\phi$ has dense range and $T$ is quasi-conjugate to the Lipschitz-free operator $T_f$ induced by $f$.
\end{proof}

\begin{example} Let $T\colon X\to X$ be the operator given by $T(x_1,x_2,\dots )=(y_k)_{k\ge 1}$ with
\[
y_k=4^k\sum_{j=0}^{k} (-1)^j\binom{k}{j} x_{j+k}, \ \ \ k\in\N , 
\]
which is the Carleman linearization of the logistic polynomial $p\colon [0,1]\to [0,1]$, $p(x):=4x(1-x)$, where 
\[
X=\{ (x_i)_i\in\R^\N \ ; \ \exists r>0 \mbox{ such that } \sup_i\abs{x_i}r^i<\infty \} .
\]
Actually, the complex version of this linearization is considered in  \cite{MP15} to show that $T$ is mixing and chaotic.
The space $X$ is an inductive limit of Banach spaces whose dual is the Fréchet space 
$$
X'=\{ (y_i)_i\in\R^\N \ : \ \sum_{i\in\N}\abs{y_i}R^i<\infty  \ \ \forall R>0\} , 
$$
and the embedding 
$j \colon [0,1]\to X$ is defined as $j(x)=(x,x^2,x^3,\dots )$. We show that 
$\overline{\spa(j(M))}=X$  by following the argument given in \cite{MP15} for the complex case: If $\phi\in X'$, $\phi=(y_i)_i$, is such that $\phi(j(x))=\sum_i y_ix^i=0$ for all $x\in M$, then $f(z):=\sum_i y_iz^i$ is an entire function that annihilates on $[0,1]$, which means that $\phi=0$, and the desired density is shown. Moreover, if $q$ is a continuous seminorm on $X$, we consider the bounded set $B:=\{ (x_i)_i\in\R^\N \ ; \  \sup_i\abs{x_i}2^{-i}\leq 1 \}$. Let $c>0$ such that $q(u)\leq c$ for all $u\in B$. If $x,y\in [0,1]$ then, 
$$
q(j(x)-j(y))=\abs{x-y}q((1,x+y,x^2+xy+y^2,\dots))\leq c\abs{x-y}, 
$$
and $j\colon [0,1]\to X$ is seminorm-Lipschitz. 
Now, by Theorem~\ref{thm:whyLF}, we know that $T$, being quasi-conjugated to the hypercyclic Lipschitz-free operator $T_p$ induced by $p$ on the Banach space $\mathcal{F}([0,1])$,  is sequentially hypercyclic. We recall that the complexification $X+iX$ of $X$ 
can be identified with $\mathcal{H}(0)$, the space of holomorphic germs at $0$, so that the associated complexification $T+iT$ of $T$ is also sequentially hypercyclic. 
\end{example}

\subsection{A note on the choice of the base point and the set $\boldsymbol{\lip_0(M,M)}$} \label{subsec:base_point}

    In this article, we deduce the dynamical properties of operators of the form $T_f$ from the specific properties of the underlying metric space $M$ around the base point $0_M$. This approach is particularly relevant since the dynamics of the operator $T_f$ can be difficult to grasp directly, whereas the map $f$ is generally simpler to understand. This leads us to first discuss the choice of the base point. Note that the space $\lip_0(M)$ and its predual $\free{M}$, as Banach spaces, do not depend on the choice of the distinguished point $0_M$. Indeed, considering two different base points $0$ and $0'$  and their corresponding spaces $\lip_0(M)$ and $\lip_{0'}(M)$, we observe that the map $\Phi_{0,0'}\colon \lip_0(M) \to \lip_{0'}(M)$ defined by $\Phi_{0,0'}(f):= f-f(0') $ is linear, isometric and onto. It maps Lipschitz functions vanishing at $0$ to those vanishing at $0'$, so it defines an isometric isomorphism between both Banach spaces. This map is also $w^*$-$w^*$-continuous and is therefore the preadjoint of a certain operator $\Phi_{0,0'*}\colon \mathcal{F}_{0'}(M) \to \mathcal{F}_0(M)$ which is again a linear onto isometry. 
    
    But, it may be worth noting that these maps may not carry any other additional structure. For instance, multiplicative or order structures (as in Banach algebras or Banach lattices theory) are often considered in the literature, and their properties are not necessarily preserved by the aforementioned isometric isomorphism, so they strongly depend on the choice of the distinguished point. The same happens for the set of Lipschitz-free operators. Indeed, the conjugation through an isomorphism gives a correspondence between operators in the two spaces, but given $T_f\in \mathcal{L}(\mathcal{F}_0(M))$, and considering $S:= (\Phi_{0,0'*})^{-1}\circ T_f\circ \Phi_{0,0'*} \in \mathcal{L}(\mathcal{F}_{0'}(M)) $  its corresponding operator through the isomorphism, although it is preserving the same dynamical properties as $T_f$, it is not guaranteed that $S$ is a Lipschitz-free operator (i.e. it may not exist $g\in \lip_{0'}(M,M)$ such that $S= T_g$). 
    
    Loosely speaking, the choice of the base point $0$ in the metric space $M$ is fixing the set of Lipschitz mappings on $M$ to consider (i.e., the set $\lip_0(M,M)$) by imposing the restriction $f(0)=0$; thus, determining the set of $T_f$ operators on $\mathcal{F}_0(M)$. 
    
    However, if we shift the focus to study a given Lipschitz map $f\colon M\to M$ having two different fixed points, the chosen base point $0$ and $0'$, then both Lipschitz-free operators operators $T_f\colon \mathcal{F}_0(M)\to \mathcal{F}_0(M)$ and $T_f\colon \mathcal{F}_{0'}(M)\to \mathcal{F}_{0'}(M)$ turn out to be topologically conjugated by $\Phi_{0,0'*}$. Thus, the dynamical properties preserved by conjugacy (such as hypercyclicity or weakly mixing)  of such $T_f$ may be equally studied through shifting the base point of $M$ to any of its fixed points. 

\subsection{Links between backward shifts on trees and Lipschitz-free operators}
\label{subsec:shifts-M}

The goal of this section consists in showing that backward shifts on trees over $\ell^1$ can be seen as Lipschitz-free operators by providing the corresponding metric space and Lipschitz map.

Following \cite{GEPa23}, 
a \textit{countable directed tree} is a pair $(V,E)$ consisting of a countable set of vertices $V$ and a set of directed edges $E\subset V\times V\backslash \{(v,v):v\in V\}$. We also assume that the graph is connected without cycles and satisfies that for every $v\in V$ there exists at most one $w\in V$ such that $(w,v)\in E$. In this case,  $w$ is called the \textit{parent} of $v$ (denoted by $w=\pare(v))$ and $v$ is a child of $w$ ($v\in \chil(w)$). Note that, due to the connectedness assumption, there is at most one vertex without a parent; if such a vertex exists, it is called the \textit{root} of the tree. 

Concerning the notation, since the set $V$ and the map $\pare$ encode all the information on the graph, we might use $V$ and $\pare$ to characterize the graphs in what follows. Also, $V$ is a countable set, so we may enumerate it by a countable set of indices, sometimes writing $V:=\{v_n\}_{n=0}^\infty$ or $V:=\{v_n\}_{n\in \ZZ}$ and in such cases, provided that $V$ is a rooted tree, we will always use $v_0$ to denote the root, while if $V$ is unrooted, $v_0$ may denote an arbitrary point of $V$.  

Given $\mu=(\mu_v)_{v\in V}\in \R^V$ a sequence of non-zero weights, the (real) $\ell^1$ space over the countable directed tree $V$ is defined as

\[\ell^1(V,\mu):= \{x\in \RR^V: \|x\|_1:= \sum_{v\in V} |x(v)\mu_v| < \infty \} \]
The unweighted case $\mu=1$ will simply be denoted as $\ell^1(V)$.

The backward shift operator over elements $x\in \RR^V$ is formally defined as 
\[(B x)(v):= \sum_{u\in \chil(v)} x(u), \quad  v\in V.\]

Then, for the metric space, take the original tree adding an artificial distinguished point, $M_V:=V \cup \{0\}$ and the distance defined for $v\ne w$ in $M_V$ by
\[
d(v,w)= |\mu_v|+|\mu_w|
\]
where we consider $\mu_0=0$. The Banach space $\free{M_V}$ is then isometrically isomorphic to $ \ell^1(V,\mu)$ through the linear isometry defined by $\Phi(\delta_v):=e_v$ for any $v\in V$.

This kind of metric is often used in the literature (see for instance \cite{ACP21} or \cite{Tap24}), as any of such metric spaces $M$ will produce Lipschitz-free spaces isometrically isomorphic to $\ell^1$---see \cite{God10}. More specifically, Proposition 1.6 from \cite{ACP21} already shows the usefulness of this representation since the backward shift operator on $\ell^1$ can be achieved as a $T_f$ operator. We extend and complement such a result by providing the corresponding metric space and Lipschitz map associated with weighted backward shifts on trees, which will be an essential pool of examples.

Now, consider the mapping $f\colon M_V\to M_V$ by
\begin{align*}
    f(v):=\begin{cases} 
      $0$ & \text{if $v$ is the root of $V$ or $v=0$}, \\
      \pare(v) & \text{otherwise.}
   \end{cases}
\end{align*}
It is easily seen that
\[\lip(f)= \sup_{v\in V\backslash\{root\}} \dfrac{|\mu_{\pare(v)}|}{|\mu_v|} ,\]
so $f\in \lip_0(M_V,M_V)$ if and only if $\sup_{v\in V\backslash\{root\}} \dfrac{|\mu_{\pare(v)|}}{|\mu_v|} <\infty$ if and only if the backward shift $B$ is continuous on $\ell^1(V,\mu)$. Thus, $T_f$ is conjugated to the backward shift $B$ on $\ell^1(V,\mu)$. 

In Linear Dynamics we can alternate between an unweighted backward shift $B$ on the weighted space $\ell^1(V,\mu)$ and  a weighted backward shift $B_\lambda$ on an unweighted space $\ell^1(V)$ where given a sequence of non-zero weights $\lambda=(\lambda_v)_{v\in V\backslash\{root\}}\subset \RR^{V\backslash\{root\}}$, the weighted backward shift operator $B_{\lambda}$ over elements $x\in \RR^V$ is formally defined as 
\[(B_\lambda x)(v):= \sum_{u\in \chil(v)}\lambda_u x(u), \quad  v\in V.\]
Indeed, both dynamical systems are topologically conjugated through the natural isometric isomorphism
\begin{equation*}
\begin{tikzcd}
\ell^1(V,\mu) \arrow[r, "B"] \arrow[d, "\phi_{\mu}"'] & \ell^1(V,\mu) \arrow[d, "\phi_{\mu}"] \\
\ell^1(V) \arrow[r, "B_{\lambda}"'] & \ell^1(V)
\end{tikzcd}
\end{equation*}
where $\phi_\mu((x_v)_{v\in V}):=(\mu_vx(v))_{v\in V}$, and the weights $\mu$ and $\lambda$ satisfy the relations
\[
\lambda_v = \frac{\mu_{\operatorname{par}(v)}}{\mu_v}, \quad \text{for }v \in V, \ v \neq \text{root}.
\]

Indeed, notice that one could alternatively achieve the same  conjugacy relation of $(\ell^1(V), B_\lambda)$ and $(\free{M_V}, T_f)$ by defining the same distance on $M_V$ by the relation
\begin{align*}
    d(v,0):=d_v=\begin{cases} 
      1 & \text{if $v=v_0$ (being root or not)} \\
      \dfrac{d_{\pare(v)}}{|\lambda_v|} & \text{otherwise,}
   \end{cases}
\end{align*}
and $d(u,w):= d(u,0)+d(w,0)$ for $u\neq w$ in $M_V$.

Notice that $d_v$ is well-defined by connectedness of the tree, and that $\free{M_V}$ is isometrically isomorphic to $ \ell^1(V)$ through the linear isometry defined by $\Phi(m_{v,0}):=e_v$. 

It is easily seen that in this case,
\[\lip(f)=\sup_{v\in V\backslash\{root\}} \dfrac{d_{\pare(v)}}{d_v}= \sup_{v\in V\backslash\{root\}} |\lambda_v|,\]
so $f\in \lip_0(M_V,M_V)$ if and only if $\sup_{v\in V\backslash\{root\}}|\lambda_v|<\infty$, if and only if the weighted backward shift $B_\lambda$ is continuous on $\ell^1(V)$.

By combining all the argument above, we can state the relation between backwards shifts on trees and Lipschitz-free operators for further reference, in its more general form.

\begin{proposition}\label{prop:backward-Tf}
    Let $V$ be a countable directed tree (with or without root) and $B_\lambda$ be a continuous weighted backward shift defined on the weighted real space $\ell^1(V,\mu)$, where $\mu$ and $\lambda$ are sequences of non-zero weights. Then, there exists a metric $M_V$ and a Lipschitz mapping $f\in\lip_0(M_V,M_V)$ such that the dynamical system $(\ell^1(V,\mu),B_{\lambda})$ is topologically conjugated to $(\free{M_V},T_f)$ through an isometric isomorphism. 
\end{proposition}

Let us note that in this correspondence between the backward shift and a Lipschitz map $f\in\lip_0(M_V,M_V)$ , we have started with a tree structure and thus a map $\textit{\pare}$ on $V$ in order to define a Lipschitz map $f\in \lip_0(M_V,M_V)$. If we change our point of view and start with an arbitrary Lipschitz map $f\in \lip_0(M_V,M_V)$ where $M_V:=V \cup \{0\}$ is endowed with the distance defined for $v\ne w$ in $M_V$ by
\[
d(v,w)=|\mu_v|+|\mu_w| \quad \text{(with $\mu_0=0$)},
\]
then $\free{M_V}$ can be identified to $ \ell^1(V,\mu)$ as before and we can consider $T_f$ as an operator on $ \ell^1(V,\mu)$, the action of this operator being given by 
\[T_f(e_v)=\begin{cases} e_{f(v)} &\text{if $f(v)\ne 0$}\\
0 &\text{if $f(v)=0$}\end{cases}.\]
We can therefore define $E=\{(w,v)\in V\times V: w=f(v)\}$. In doing so, $T_f$ can be seen as the backward shift on the graph $(V,E)$. However it is important to notice that $(V,E)$ may not be a directed tree since $(V,E)$ is maybe not connected and since $(V,E)$ may contain cycles. Nevertheless we keep the property that for every $v\in V$, there exists at most one $w\in V$ such that $(w,v)\in E$.\\

The dynamical properties of shifts on trees have been studied through several papers \cite{Ma17, GEPa23, AA25, LP25, GePapreprint, AC26}. In particular, weakly mixing shifts on trees have been characterized by Grosse-Erdmann and Papathanasiou \cite{GEPa23}.

\begin{theorem}[\cite{GEPa23}]\label{caracshift}
Let $(V,E)$ be a countable directed tree and $\mu$ a sequence of positive weights on $V$. Suppose that the backward shift $B$ is continuous on $\ell^1(V,\mu)$. Then
\begin{enumerate}
\item If $V$ is a rooted tree then the following assertions are equivalent:
\begin{itemize}
\item $B$ is hypercyclic on $\ell^1(V,\mu)$,
\item $B$ is weakly mixing on $\ell^1(V,\mu)$,
\item there exists an increasing sequence $(n_k)_k$ of positive integers such that, for each $v\in V$,
\[\inf_{u\in \text{Chi}^{n_k}(v)}\mu_u\xrightarrow[k\to +\infty]{} 0.\]
\end{itemize}
\item If $V$ is an unrooted tree then the following assertions are equivalent:
\begin{itemize}
\item $B$ is hypercyclic on $\ell^1(V,\mu)$,
\item $B$ is weakly mixing on $\ell^1(V,\mu)$,
\item there exists an increasing sequence $(n_k)_k$ of positive integers such that, for each $v\in V$,
\[\inf_{u\in \text{Chi}^{n_k}(v)}\mu_u\xrightarrow[k\to +\infty]{} 0 \quad \text{and}\quad \min\left(\mu_{\pare^{n_k}(v)},\inf_{u\in \text{Chi}^{n_k}(\pare^{n_k}(v))}\mu_u\right)\xrightarrow[k\to +\infty]{} 0.\]
\end{itemize}
\end{enumerate}
\end{theorem}

\section{Blow-up/collapse behaviour}

Throughout this section we assume that $T\colon X\to X$ is a linearization on a topological vector space $X$ of a continuous map $f\colon K\to K$  on a Hausdorff topological space $K$ with certain fixed point $x_0\in K$. In the Lipschitz-free context, the space $K$ will be our metric space $M$ and  $X=\mathcal{F}(M)$.

There are several criteria (sufficient conditions) for hypercyclicity/transitivity of linear operators (see, e.g. \cite{GEP11}). All the versions involve a blow-up/collapse around $0$. We recall that a linear operator $T$ is said to be \textit{weakly mixing} if the product operator $T\times T$ is topologically transitive. We will recall the following general formulation: 

\begin{theorem}[Hypercyclicity Criterion]\label{thm:HC}
Let $X$ be a topological vector space, $T\in \mathcal{L}(X)$ and $(n_k)$ an increasing sequence.
If there exist dense subsets $X_0,X_1$ in $X$ such that
\begin{enumerate}
\item for each $x\in X_0$ there exists a sequence $(x_k)_k\in X^{\mathbb{N}}$ such that $\lim_{k\to \infty} x_k=x$ and $\lim_k T^{n_k}x_{k}=0$,
\item for each $y\in X_1$ there exists a sequence $(y_k)_k\in X^{\mathbb{N}}$ such that $\lim_{k\to \infty} y_k=0$ and $\lim_k T^{n_k}y_{k}=y$,
\end{enumerate}
then $T$ is weakly mixing. Moreover, if $T$ satisfies the Hypercyclicity Criterion with respect to the full sequence $(k)$ then $T$ is mixing.
\end{theorem}

If $T\colon X\to X$ is a linearization of a map $f$ then a natural way to get that $T$ satisfies the Hypercyclicity Criterion is to require that $f$ has a {blow-up/collapse} property around the fixed point $x_0$ such that $j(x_0)=0$.

\begin{definition}\label{BUChcnl}
We say that a continuous map $f\colon K\to K$ on a topological space $K$ with $x_0\in K$ so that $f(x_0)=x_0$ has a \textbf{blow-up/collapse} property around $x_0$ with respect to a strictly increasing sequence $(n_k)_k\in \N^\N$ if there exist $K_0,K_1\subset K$ dense subsets such that
\begin{enumerate}
\item[(i)] for each $x\in K_0$, there exists a sequence $(x_k)_k\in K^{\N}$ such that $\lim_{k\to\infty} x_k=x$ and $\lim_{k\to\infty} f^{n_k}(x_k)=x_0$. 

\item[(ii)] for each $y\in K_1$ there exists a sequence $(y_k)_k\in K^\N$ such that $\lim_{k\to\infty} y_k=x_0$ and $\lim_{k\to\infty} f^{n_k}(y_k)=y$. 
\end{enumerate}
\end{definition}

\begin{examples}\label{exa1}
We recall that a map $f\colon K\to K$ is \emph{locally eventually onto} (sometimes called \textit{topologically exact}) if for any non-empty open set $U\subset K$, there exists $n\in\N$ such that $f^n(U)=K$. If, moreover, $f$ admits a fixed point $x_0$ then $f$ trivially satisfies the blow-up/collapse property around $x_0$ with respect to the full sequence $(k)_k\subset \NN$. This is the case for:
\begin{enumerate}
    \item the tent map $t\colon [0,1]\to [0,1]$ given by $t(x)=2x$ for $x\in [0,1/2]$, $t(x)=2-2x$ for $x\in [1/2,1]$;
    \item  the logistic map $l_4\colon[0,1]\to [0,1]$ defined by $l_4(x)=4x(1-x)$.
\end{enumerate}
Actually, much less is needed: If $f(x_0)=x_0$ is so that for every neighbourhood $U$ of $x_0$ we have that there is $n\in\NN$ with $f^n(U)=K$ and $\bigcup_{n\in\NN} f^{-n}(\{x_0\})$ is dense in $K$, we easily have that $f$ satisfies the blow-up/collapse property around $x_0$ with respect to the full sequence $(k)_k$. For instance, the anti-symmetric double tent map $h\colon[-1,1]\to [-1,1]$ defined by $h(x)=t(x)$ for $x\in [0,1]$ and $h(x)=-t(-x)$ for $x\in [-1,0]$ satisfies this property but $h$ is not even transitive (see Example 4.6 in \cite{ACP21}). 
\end{examples}

The blow-up/collapse property above is more general than the Hypercyclicity Criterion for Lipschitz-free operators given in \cite{ACP21}. 

\begin{definition}(\cite[Definition 3.2]{ACP21})
   Let $M$ be a pointed metric space and $f\in \lip_0(M,M)$. It is said that $f$ satisfies the \textbf{HCL} (Hypercyclicity Criterion for Lipschitz-free operators) \textbf{around a fixed point} $x_0$ if there exists an increasing sequence $(n_k)_k$, two dense sets $M_0,M_1\subset M$ and a sequence of maps $g_{n_k}\colon M_1 \to M$ such that 
\begin{enumerate}
    \item $f^{n_k}(x)\xrightarrow{k\to \infty}x_0$ for every $x\in M_0$;
    \item $g_{n_k}(y)\xrightarrow{k\to \infty}x_0$ for every $y\in M_1$;
     \item $(f^{n_k}\circ g_{n_k})(y)\xrightarrow{k\to \infty}y$ for every $y\in M_1$.
\end{enumerate}
\end{definition}

Indeed, given a metric space $M$ and a map $f\in \lip_0(M,M)$ with a fixed point $x_0$, considering the inverse mappings $g_{n_k}\colon M_1\to M_1,$ defined by $g_{n_k}(y):=y_k$ for every $y\in K_1$, then $f$ satisfying the HCL around $x_0$ clearly implies having the blow-up/collapse property around $x_0$. 
It is thus not surprising that we get the following result.

\begin{proposition}
Let $T\colon X\to X$ be a linearization on a topological vector space $X$ of a continuous map $f\colon K\to K$ with a fixed point $x_0$. If $f$ has a blow-up/collapse property around $x_0$ with respect to $(n_k)_k$ then $T$ satisfies the Hypercyclicity Criterion with respect to $(n_k)_k$ and, in particular, $T$ is weakly mixing (respectively, weakly mixing and hypercyclic if $X$ is completely metrizable and separable). Moreover, if $(n_k)_k=(k)_k$, $T$ is also mixing. 
\end{proposition}

\begin{proof}
    Given $K_0,K_1\subset K$ dense subsets associated to the blow-up/collapse property of $f$ around $x_0$, we define $X_0=\spa (j(K_0))$ and $X_1=\spa (j(K_1))$. By linearity of $T$ and density in $X$ of $\spa (j(K))$, we immediately get that $T$ satisfies the Hypercyclicity Criterion stated in Theorem~\ref{thm:HC}. We then obtain that $T$ is weakly mixing. In the particular case where $X$ is completely metrizable (i.e., an $F$-space) and separable, we have that $T$ is hypercyclic. The mixing property when $(n_k)_k=(k)_k$ is obviously obtained. 
\end{proof}

Unfortunately, the blow-up/collapse property of $f$ around $x_0$ is too strong to characterize the weak mixing nature of the linearization $T_f$. Here are two important examples: 

\begin{examples}\label{mainexamples}
\begin{enumerate}
\item  By following the notation of Subsection \ref{subsec:shifts-M}, we consider $V=\mathbb{Z}\times \mathbb{Z}^+$ and the map
\[\pare((m,n))=\begin{cases}
(m,n-1) &\text{if $n\ge 1$}\\
(m-1,0) &\text{if $n=0$}.
\end{cases}.\]
If we let $\mu_{(m,n)}=2^{-n}$ then the map $f$ defined on $M_V=V\cup \{0\}$ by $f(v)=\pare(v)$ if $v\in V$ and $f(0)=0$ belongs to $\lip_0(M_V,M_V)$. The Lipschitz-free operator $T_f$ is then the backward shift on $\ell^1(V,\mu)$ and is weakly mixing by Theorem~\ref{caracshift}. However, $f$ fails the blow-up/collapse property around $0$ as condition (i) of Definition \ref{BUChcnl} is not satisfied. Indeed, no orbit of a non-zero point under the action of $f$ goes to $0$.

\begin{figure}[h!]\label{fig:ring-tree}
\centering
\includegraphics[width=0.7\textwidth]{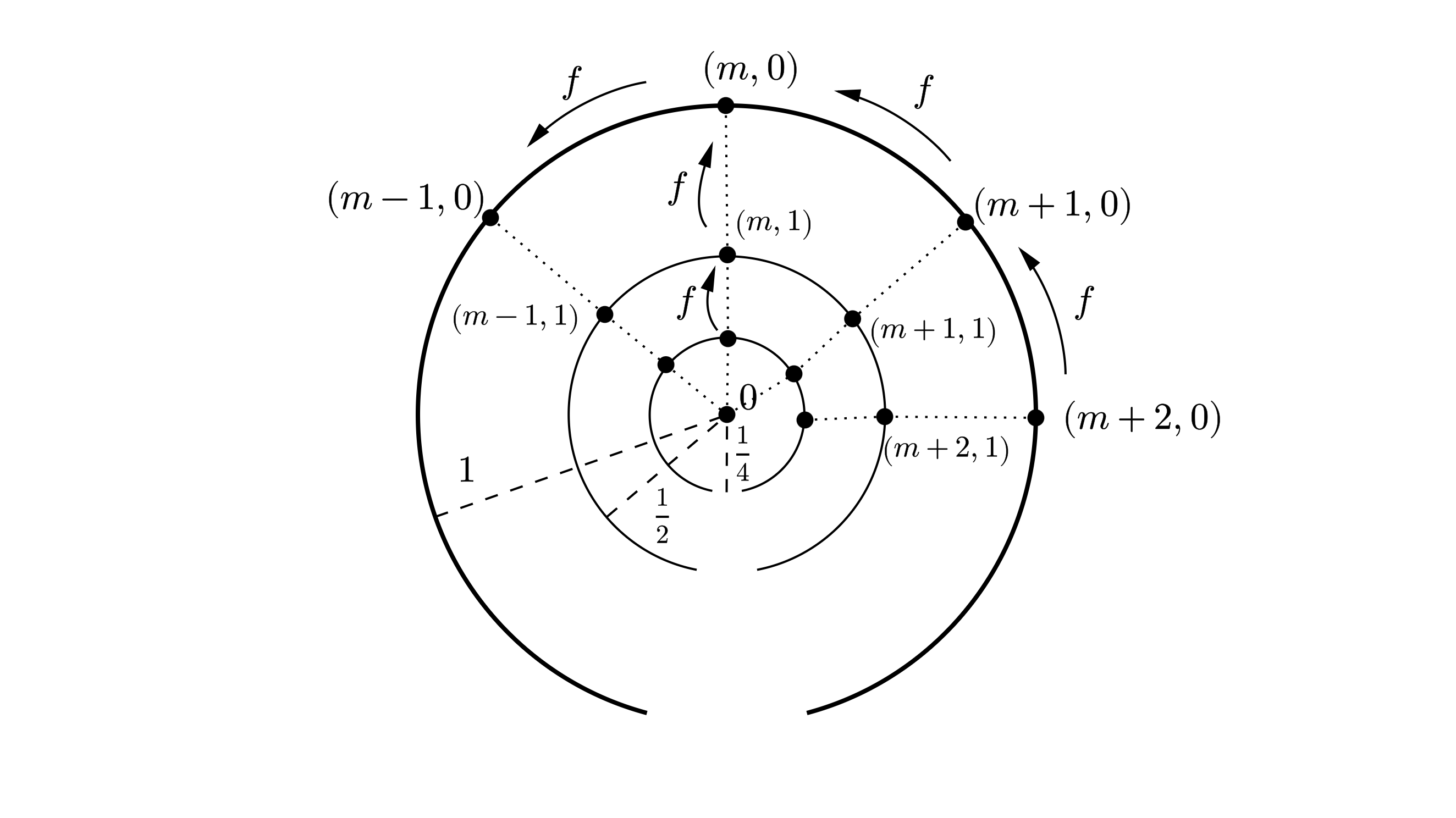}
\caption{Dynamical system $(M_V,f)$ not having the blow-up/collapse property around $0$, but its linearization $T_f$ being weakly mixing.}
\end{figure}

\item The second example, already used in \cite[Example 3.5]{ACP21}, is given by the map $f(0)=0$, $f(1/n)=1/(n-1)$ and $f(1)=1$ on the metric space $M=\{0\}\cup\{1/n:n\ge 1\}$, that has two fixed points ($0$ and $1$). It has been proven in \cite[Example 3.5]{ACP21} that $T_f$ is weakly mixing. However $f$ does not have the blow-up/collapse property around either of its fixed points. Indeed, $f$ fails the blow-up/collapse property around $0$ as condition (i) of Definition \ref{BUChcnl} is not satisfied, since every orbit (except the orbit of $0$) is tending to $1$.  Analogously, $f$ fails the blow-up collapse behaviour around $1$, as it fails condition (ii). Since $1$ is an isolated point, a sequence $(y_k)_k$  tending to $1$ is eventually equals to $1$ and thus we always have $\lim_{k\to\infty} f^{n_k}(y_k)=1$.   \\
It should be noted that in \cite[Remark 3.6]{ACP21}, it was stated that the function $f$ above satisfied the HCL for the distinguished point $1\in M$, which turns out to be not true as mentioned above.
\end{enumerate}
\end{examples}

In view of these two examples, we would like to weaken the blow-up/collapse property.

 \section{The targeting property}

 Inspired by the characterization stated in Theorem~\ref{caracshift} for the unrooted tree, we introduce in this section a new condition, called the {targeting property}, that will allow us this time to cover both examples in Examples~\ref{mainexamples}. Moreover, we will then show that on the metric spaces considered in these examples, the targeting property of $f$ characterizes the hypercyclicity of the linearization $T_f$. As a by-product, we will get that these linearizations $T_f$ are hypercyclic if and only if they are weakly mixing.

The characterization stated in Theorem~\ref{caracshift} for the unrooted tree relies on the existence of an increasing sequence $(n_k)_k$ of positive integers such that, for each $v\in V$,
\[\inf_{u\in \text{Chi}^{n_k}(v)}\mu_u\xrightarrow[k\to +\infty]{} 0 \quad \text{and}\quad \min\left(\mu_{\pare^{n_k}(v)},\inf_{u\in \text{Chi}^{n_k}(\pare^{n_k}(v))}|\mu_u|\right)\xrightarrow[k\to +\infty]{} 0.\]
In particular, for the backward shift on trees considered in Examples~\ref{mainexamples}, the hypercyclicity comes from the fact that, for each $v\in V$,
\[\inf_{u\in \text{Chi}^{k}(v)}\mu_u \xrightarrow[k\to +\infty]{} 0 \quad \text{and}\quad \inf_{u\in \text{Chi}^{k}(\pare^{k}(v))}|\mu_u| \xrightarrow[k\to +\infty]{} 0,\]
while $\mu_{\pare^{k}(v)}$ is ultimately equal to $1$. In fact, in the characterization of hypercyclicity for the weighted shifts on an unrooted tree, the convergence of $(\mu_{\pare^{n_k}(v)})_k$ to $0$ allows to say that the sequence $(B^{n_k}e_v)_k$ tends to $0$ while the convergence of $(\inf_{u\in \text{Chi}^{n_k}(\pare^{n_k}(v))}|\mu_u|)_k$
 to $0$ allows to find a sequence $(v_k)_k$ such that $(\mu_{v_k})_k$ tends to $0$ and $B^{n_k}(e_v-e_{v_k})=0$. In other words, it is not required that $e_v$ collapses around $0$. Instead, we can require the weaker condition of being eventually arbitrarily close to orbits coming from points close to $0$. Analogously, one can weaken the blow-up condition too.  This is the idea behind the targeting property stated below. We will define it in the general context of uniform spaces  since this is the natural class that contains the cases in which we are mostly interested, namely metric spaces and subsets of topological vector spaces. 

\begin{definition}\label{def:targeting}
We say that a continuous map $f\colon Y\to Y$ on a uniform  space $(Y,\mathcal{U})$ with $x_0\in Y$ so that $f(x_0)=x_0$ has the \textbf{targeting property} around $x_0$ if, for every $U\in \mathcal{U}$ in the uniformity, and for each finite family $F\subset Y$, we find $n\in \NN$ such that, for every $x,y\in F$,  
\begin{enumerate}
\item[(i)]  there exist $u,u_0\in Y$ with 
\[
\{(u,x),(u_0,x_0),(f^n(u),f^n(u_0))\} \subset U, \ \mbox{ and}
\]
\item[(ii)] there exist $v,v_0\in Y$ with 
\[
\{(v,v_0),(f^n(v),y),(f^n(v_0),x_0)\} \subset U.  
\] 
\end{enumerate}
If the set of $n\in\NN$ satisfying the above conditions is co-finite, the we say that $f$ satisfies the \textbf{strict targeting property}. 
\end{definition}

\begin{remark} The two main cases in which we are interested to apply the targeting property are the following: 
\begin{enumerate}
\item[(a)] If $Y$ is a subset of a topological vector space $X$, with the inherited topology, then a continuous map $f\colon Y\to Y$ that fixes $0$ has the targeting property around $0$ if, and only if, for every $0$-neighbourhood $U$ in $X$ and for each finite family $F\subset Y$, we find $n\in \NN$ such that, for every $x,y\in F$,  
\begin{enumerate}
\item[(i)]  there exist $u,u_0\in Y$ with 
\[
u\in x+U, \ \ u_0\in U, \ \ f^n(u)- f^n(u_0)\in U, \ \mbox{ and}
\]
\item[(ii)] there exist $v,v_0\in Y$ with 
\[
v-v_0\in U, \ \ f^n(v)\in y+U, \ \ f^n(v_0)\in  U.  
\] 
\end{enumerate}

\item[(b)] When $Y$ is a metric space $(M,d)$ the targeting property is equivalent to: For any $\varepsilon>0$, and for each finite family $F\subset M$, we find $n\in \NN$ such that, for every $x,y\in F$,   
\begin{enumerate}
\item[(i)]  there exist $u,u_0\in M$ with 
\[
\max \{d(u,x),d(u_0,x_0),d(f^n(u),f^n(u_0))\} <\varepsilon, \ \mbox{ and}
\]
\item[(ii)] there exist $v,v_0\in M$ with 
\[
\max \{d(v,v_0),d(f^n(v),y),d(f^n(v_0),x_0)\} <\varepsilon.  
\] 
\end{enumerate}
    If, moreover, $(M,d)$ is separable, then $f$ has the targeting property around $x_0$ if, and only if, there exist a dense subset $M_0\subset M$ and an increasing sequence $(n_k)_k$ of integers such that
    \begin{enumerate}
\item[(i)]  for each $x\in M_0$ there exist sequences $(x_k)_k$ and $(z_k)_k$ in $M^\NN$  with $\lim_k x_k=x$, $\lim_kz_k=x_0$, $\lim_k d(f^{n_k}(x_k),f^{n_k}(z_k))=0$, and 
\item[(ii)] for each $y\in M_0$ there exist sequences $(y_k)_k$ and $(w_k)_k$ in $M^\NN$  with $\lim_k d(y_k,w_k)=0$,  $\lim_kf^{n_k}(y_k)=y$, and $\lim_k f^{n_k}(w_k)=x_0$. 
\end{enumerate}
Indeed, the above property clearly implies the targeting property around $x_0$. For the converse, let $M_0\subset M$ be a countable dense subset. We write $M_0$ as a countable union of an increasing sequence of finite sets $M_0=\bigcup_k F_k$. By the targeting property, let $(n_k)_k$ be a sequence of integers (which we assume increasing) associated to the sequence of finite sets $(F_k)_k$ and the sequence of distances $(1/k)_k$. Given $x\in M_0$, we find $k_0\in\NN$ so that $x\in F_{k_0}$ and, by condition (i), there exist $u(k),u_0(k)\in M$ with 
\[
\max \{d(u(k),x),d(u_0(k),x_0),d(f^{n_k}(u(k)),f^{n_k}(u_0(k)))\} <1/k, \ k\geq k_0.
\]
By setting $x_k=u(k)$ and $z_k=u_0(k)$ we obtain our new condition (i). Analogously, given $y\in M_0$ one can find the sequences $(y_k)_k$ and $(w_k)_k$ so that the new condition (ii) is satisfied.\\
We can also remark that in this context, $f$ has the targeting property around $x_0$ if and only if there exist two dense subsets $M_0,M_1\subset M$ and an increasing sequence $(n_k)_k$ of integers such that (i) holds for each $x\in M_0$ and (ii) holds for each $y\in M_1$.

\end{enumerate}
\end{remark}

We notice that the two dynamical systems considered in Examples~\ref{mainexamples} have the strict targeting property. For instance, the dynamical system given by $f(0)=0$, $f(1/n)=1/(n-1)$ and $f(1)=1$ on $M=\{0\}\cup\{1/n:n\ge 1\}$ satisfies the targeting property around $0$ with respect to the full sequence $(k)_k$ if we consider $M_0=M\backslash\{0\}$ with $x_k=x$, $z_k=1/k$, $y_k=f^{-k}(y)$ if $y\neq 1$ (otherwise, $y_k=1/k$), and  $w_k=0$. For the dynamical system  defined on $M=\ZZ\times\ZZ^+\cup\{0\}$, given $x\in M_0:=M\setminus \{0\}$, $x=(m,n)$, we set $x_k=x$ and $z_k=(m+n-k,k)$, $k\in\NN$. It is clear that $f^k(x_k)=f^k(z_k)$ for $k\geq n$, and condition (i) is satisfied. Given $y\in M_0$, we consider $y_k=y+(0,k)$ and $w_k=0$, $k\in\NN$, and (ii) is clearly fulfilled. These examples have in common that the part of the collapse fails around the base point, but the blow-up is satisfied. One can also construct an example where both, the collapse and the blow-up, fail around the base point, but the dynamical system has the targeting property. 

\begin{example}\label{exnoblow}
    Let $M:=\{0,1\}\cup\{1/n \ : \ n>2\}\cup \{n/(n+1) \ : \ n>1\}$, with the usual distance, and let $f:M\to M$ be defined as $f(0)=0$, $f(1)=1$, $f(1/(n+1))=1/n$, $n>2$, $f(1/3)=1$, $f(n/(n+1))=(n-1)/n$, $n>2$, and $f(2/3)=0$. We easily observe that the orbits of points $0<x\leq 1/3$ are eventually $1$, and the orbits of points $2/3\leq y<1$ are eventually $0$, so none of the two fixed points of $f$ has either collapse or blow-up around it. On the other hand, this system has the targeting property around $0$ (or around $1$). We set $M_0=M\setminus \{0,1\}$. If $x\in M_0\cap [0,1/3]$, then we select $x_k=x$ and $z_k=1/(k+2)$, $k\in\NN$, and if $x\in M_0\cap [2/3,1]$, then we select $x_k=x$ and $z_k=0$, $k\in\NN$, to obtain condition (i). For the second condition, if $y\in M_0\cap [0,1/3]$, then $y_k:=f^{-k}(y)$, $w_k:=0$, $k\in\NN$, and if $y\in M_0\cap [2/3,1]$, then $y_k:=f^{-k}(y)$, $w_k:=(k+1)/(k+2)$, $k\in\NN$, which yields the targeting property. 
\end{example}

Going back to the logistic map, which is part of the logistic family, we obtain new interesting examples. 

\begin{example}\label{logisticf}
    Let $p_\lambda\colon [0,1]\to [0,1]$ be defined as $p_\lambda (x)=\lambda x(1-x)$, $0\leq \lambda\leq 4$. 
    If $\lambda>1$ one finds a countable subset $M_0\subset [0,1]$ such that, for $M:=\{0\}\cup \{x_\lambda\}\cup M_0$, we have $p_\lambda (M)\subset M$, $f:=p_\lambda|_M$ satisfies the targeting property with respect to the full sequence $(k)_k$. Indeed, for $1<\lambda<3$ we have that $0$ is a repelling fixed point and $x_\lambda:=(\lambda-1)/\lambda$ is an attractive fixed point. By taking any point $0<y<x_\lambda$, the set $M_0$ consisting of the forward orbit of $y$, that converges to $x_\lambda$, and a backward orbit of $y$ converging to $0$ yields the targeting property. In case that $3\leq \lambda\leq 4$, we just take as $M_0$ a backward orbit of $x_\lambda$ converging to $0$. We easily have that the operator given by its {Carleman linearization} $S_\lambda \colon X \to X$, $X=\{ (x_i)_i\in\R^\N \ ; \ \exists r>0 \mbox{ with } \sup_i\abs{x_i}r^i<\infty \}$, is given by $S_{\lambda}(x_1,x_2,\dots )=(y_k)_{k\ge 1}$ with
$$
y_k=\lambda^k\sum_{j=0}^k (-1)^j\binom{k}{j}x_{j+k}, \ \ k\in \N. 
$$
\end{example}

The targeting property is strictly weaker than the blow-up/collapse property. However, assuming that $f$ has the targeting property and the linearization $T$ is so that the embedding $j$ is uniformly continuous will be enough to ensure  that $T$ is weakly mixing.

\begin{theorem}  \label{thm:aim-WHC}
Let $T\colon X\to X$ be a linearization on a topological vector space $X$ of a continuous map $f\colon Y\to Y$ on a uniform space $Y$ with fixed point $x_0$ such that the embedding $j\colon Y\to X$ is uniformly continuous. We have: 
\begin{enumerate}
    \item If $f$ has the targeting property around $x_0$ then $T$ is weakly mixing.
    \item If $f$ has the strict targeting property anon-empty open sets round $x_0$ th en $T$ is mixing.
\end{enumerate}
\end{theorem}
\begin{proof}
We will prove the weak mixing case. Let us recall that the return set from $U$ to $V$ is defined as $N(U,V):=\{n\in\N:T^n(U)\cap V\neq \emptyset\}$ and that $T$ will be weakly mixing provided that for every non-empty open sets $U,V\subset X$ and any $0$-neighbourhood $W$, we have $N(U,W)\cap N(W,V)\neq \emptyset$. Let us then take two non-empty open sets $U,V\subset X$ and a $0$-neighbourhood $W$.

We fix $x=\sum_{i=1}^ka_ij(x_i)\in U$ and $y=\sum_{i=1}^k b_ij(y_i)\in V$ with $F:=\bigcup_{i=1}^k\{x_i,y_i\}\subset Y$. We find now a $0$-neighbourhood $W_0\subset X$ such that 
\[
\sum_{i=1}^ka_i W_0\subset W, \ \ \ \sum_{i=1}^kb_i W_0\subset W,
\]
\[
x+\sum_{i=1}^ka_i (W_0-W_0)\subset U, \ \ y+\sum_{i=1}^kb_i (W_0-W_0)\subset V. 
\]
The uniform continuity of $j\colon Y\to X$ yields the existence of $A\in \mathcal{U}$, the uniformity of $Y$, such that $(j\times j)(A)\subset \{(z_1,z_2)\in X\times X \ ; \ z_1-z_2\in W_0\}$. By the targeting property of $f$ around $x_0$, there exists $n\in\NN$ associated to $A$ and $F$ such that 
\begin{enumerate}
\item[(i)]  there exist $u(i),u_0(i)\in Y$, $i=1, \dots,k$, with 
\[
\{(u(i),x_i),(u_0(i),x_0),(f^n(u(i)),f^n(u_0(i)))\} \subset A, \ \mbox{ and}
\]
\item[(ii)] there exist $v(i),v_0(i)\in Y$, $i=1, \dots,k$, with 
\[
\{(v(i),v_0(i)),(f^n(v(i)),y_i),(f^n(v_0(i)),x_0)\} \subset A. 
\]
\end{enumerate}
By the uniform continuity of $j$, the condition $\{(u(i),x_i),(u_0(i),x_0)\}\subset A$ implies that $j(u(i))-j(x_i)\in W_0$ and $j(u_0(i))-j(x_0)=j(u_0(i))\in W_0$. Therefore, $j(u(i))-j(x_i)-j(u_0(i))\in W_0-W_0$.

We set $\bar{x}=\sum_{i=1}^ka_i (j(u(i))-j(u_0(i)))$. By (i) we have 
\[
\bar{x}=x+\sum_{i=1}^ka_i (j(u(i))-j(x_i)-j(u_0(i)))\in x+\sum_{i=1}^ka_i (W_0-W_0)\subset U, 
\]
\[
T^n(\bar{x})=\sum_{i=1}^ka_i (j(f^n(u(i)))-j(f^n(u_0(i))))\in \sum_{i=1}^ka_i W_0\subset W. 
\]
We consider now $\bar{y}=\sum_{i=1}^kb_i (j(v(i))-j(v_0(i)))$. By (ii) we obtain
\[
\bar{y}=\sum_{i=1}^kb_i (j(v(i))-j(v_0(i)))\in \sum_{i=1}^kb_i W_0\subset W, 
\]
\[
T^n(\bar{y})=y+\sum_{i=1}^kb_i (j(f^n(v(i)))-j(y_j)-j(f^n(v_0(i))))\in \sum_{i=1}^kb_i (W_0-W_0)\subset V. 
\]
That is, $n\in N(U,W)\cap N(W,V)$, as we wanted to show. 
\end{proof}

As a consequence of this theorem, all the examples considered in \ref{mainexamples}, \ref{exnoblow} and \ref{logisticf} yield a mixing and sequentially hypercyclic linearization. 

Our next goal will be to prove that, for a natural family of metric spaces $M$, the \aim{} property is exactly the property characterizing hypercyclicity and weakly mixing for the associated Lipschitz-free operators.

\subsection{When $\mathcal{F}(M)$ admits a Schauder basis given by deltas}

A natural way to better understand a Lipschitz-free space is to identify a Schauder basis. Moreover, since a $T_f$ operator maps $\dd_x$ into $\dd_{f(x)}$, in our case it becomes of special interest to get a Schauder basis made of deltas. The condition for characterizing such spaces was already introduced in \cite{AlPer23}.

\begin{definition}[{\cite[Definition 6.2]{AlPer23}}] Let $(M,d)$ be a metric space, $x_0\in M$. We say that $M$ is \textbf{radially discrete} (with respect to $x_0$) if there exists $K\geq 1$ such that for every elements $x\neq y\in M$,
\[d(x,x_0)\leq Kd(x,y).\]
\end{definition}

The fact that this property is the one characterizing the existence of a Schauder basis of deltas can already be deduced from  \cite[Theorem 6.2]{AlPer23}. Below we provide a direct argument.

We can already remark that if a sequence $(\delta_{x_n})_{n\ge 1}$ is a Schauder basis of $\mathcal{F}(M)$ then the sequence $(x_n)_{n\ge 1}$ is dense in $M\backslash\{0\}$. Indeed, if there exists $a$ in $M\backslash\{0\}$ that does not belong to the closure of $(x_n)_{n\ge 1}$ then there exists $f_a\in \lip_0(M,M)$ such that $f_a(a)=1$ and $f_a(x_n)=0$ for every $n\ge 1$ so that for every series $\sum_{n=1}^{\infty}\aa_n \delta_{x_n}$, we have
\[\|\delta_a-\sum_{n=1}^{\infty}\aa_n \delta_{x_n}\|\ge \frac{1}{\text{Lip}(f_a)}>0.\]

Recall that a sequence $(\delta_{x_n})_{n\ge 1}$ is a Schauder basis of $\mathcal{F}(M)$ if and only if for every $\mu\in \mathcal{F}(M)$, there exists a unique sequence $(\aa_n)_{n\ge 1}$ such that $\mu= \sum_{n=1}^{\infty}\aa_n \delta_{x_n}$, or equivalently, if and only if
\begin{enumerate}
\item $\text{span}\{\delta_{x_n}:n\ge 1\}$ is dense in $\mathcal{F}(M)$;
\item there exists $K\ge 1$ such that for all $m,p\ge 1$, for any sequence $(\aa_n)_{n\ge 1}\subset \RR$,
\[K \|\sum_{n=1}^{m+p}\aa_n \delta_{x_n}\|\ge \|\sum_{n=1}^{m}\aa_n \delta_{x_n}\|.\]
\end{enumerate}
These conditions will give us the following equivalences.

\begin{proposition}\label{prop:deltas-basis}
    Let $(M,d)$ be a metric space with a distinguished point $0$ and $(x_n)_{n\geq 1}$ a dense set in $M\backslash\{0\}$. Then, the following are equivalent:
    \begin{itemize}
        \item[(i)] The sequence $(\delta_{x_n})$ is a Schauder basis for $\mathcal{F}(M)$;
        \item[(ii)] There exists $K\geq 1$ such that for every $n\ge 1$, for all $y\in M\backslash\{x_n\}$, 
        \[ d(x_n,0)\leq K d(x_n,y).\]
    \end{itemize}
\end{proposition}
\begin{proof}
We first show that (ii) implies (i). 

Since $(x_n)_{n\geq 1}$ is a dense sequence in $M\backslash\{0\}$, we have that $\overline{\text{span}}\{\delta_{x_n}:n\ge 1\}=\mathcal{F}(M)$. It is thus sufficient to show that there exists $K\ge 1$ such that for all $m,p\ge 1$,  for any sequence $(\aa_n)_{n\ge 1}\subset \mathbb{R}$,
\[K \|\sum_{n=1}^{m+p}\aa_n \delta_{x_n}\|\ge \|\sum_{n=1}^{m}\aa_n \delta_{x_n}\|.\]

Let $M_m:=\{x_n\}_{n=1}^m$ and the map
\begin{equation}\label{eq:retrac}
    f_{M_m}(y):=\begin{cases}
      y &\text{ if }\ y\in M_m, \\
  0 &\text{ otherwise }.
  \end{cases}
\end{equation}
We get 
\begin{align*}
\lip(f_{M_m})&=\max_{x\in M_m}\sup_{y\ne x}\frac{d(f_{M_m}(x),f_{M_m}(y))}{d(x,y)}\\
&\le \max\left\{1,\max_{n\in\{1,...,m\}}\sup_{y\ne x_n}\frac{d(x_n,0)}{d(x_n,y)}\right\}\\
&\le K.
\end{align*}

Let $m,p\ge 1$ and $(\aa_n)\subset \mathbb{R}$.
Let $f\in \text{Lip}_0(M)$ with $\lip(f)=1$ and $ \tilde{f}=f\circ f_{M_m}$.
We then have $\lip(\tilde{f})\le K$ and
\[
K\left \|\sum_{n=1}^{m+p}\aa_n\delta_{x_n} \right \|\geq \left |\sum_{n=1}^{m+p}\aa_n  \delta_{x_n}(\tilde{f})\right |= \left|\sum_{n=1}^{m+p}\aa_n\tilde{f}(x_n)\right|=\left|\sum_{n=1}^{m}\aa_n f(x_n)\right|.
\]
We deduce that we have the desired inequality.\\

We now prove that (i) implies (ii). Indeed, if (ii) does not hold, then for every $n\geq 1$, there exists $j_n\ge 1$ such that $B(x_{j_n},\frac{d(x_{j_n},0)}{n})\backslash\{x_{j_n}\}$ is a non-empty open set. By the density of $\{x_i\}_{i=1}^\infty$, there exists $k_n\ne j_n$ such that
\[x_{k_n}\in B\left(x_{j_n},\frac{d(x_{j_n},0)}{n}\right),\]
deducing that $nd(x_{j_n},x_{k_n})<d(x_{j_n},0).$

Now, it follows that for every $n\geq 1$, 
\[\|\delta_{x_{j_n}}\|=d(x_{j_n},0)\ge n d(x_{j_n},x_{k_n}) = n \|\delta_{x_{j_n}}-\delta_{x_{k_n}}\|\]
and
\[\|\delta_{x_{k_n}}\|\ge \|\delta_{x_{j_n}}\|- \|\delta_{x_{j_n}}-\delta_{x_{k_n}}\|\ge (n-1)\|\delta_{x_{j_n}}-\delta_{x_{k_n}}\|.\]
The sequence $\{\delta_{x_n}\}_{n=1}^\infty$ is thus not a Schauder basis for $\mathcal{F}(M)$ since these inequalities would contradict the uniform boundedness of the canonical projections.
\end{proof}

Observe that condition (ii) implies that every point $x_n$ is isolated. This, plus the density of the sequence $(x_n)_{n\ge 1}$, forces the set $M$ to be countable itself and the dense sequence $(x_n)$ to be an enumeration of $M\backslash\{0\}$.

We can, therefore, state the following corollary with the complete characterization.

\begin{corollary}\label{cor:delta-basis}
Let $(M,d)$ be metric space with a distinguished point $0$. The space $\mathcal{F}(M)$ admits a Schauder basis given by deltas if and only if $M$ is at most countable and radially discrete with respect to $0$.
Moreover, such a basis is given by the deltas of any  enumeration of $M\backslash\{0\}$, and it is unconditional, since the norm on $\mathcal{F}(M)$ is equivalent to the norm 
\[\left\|\sum_{x\in M\backslash\{0\}}\aa_x \delta_{x}\right\|_1=\sum_{x\in M\backslash\{0\}}\aa_x d(x,0).\]
In particular, $\free{M}$ is isomorphic to $\ell^1(\NN)$.
\end{corollary}
\begin{proof}
The first part is clear, since it is a restatement of Proposition \ref{prop:deltas-basis} and the comment after it. 
Moreover, if $(\delta_{x_n})$ is a Schauder basis then the norm on $\mathcal{F}(M)$ is equivalent to the norm $\|\cdot\|_1$ defined by
\[\left\|\sum_{n=1}^{\infty}\aa_n \delta_{x_n}\right\|_1=\sum_{n=1}^{\infty}\aa_n d(x_n,0).\]
Indeed, given $\mu=\sum_{n=1}^{\infty}\aa_n \delta_{x_n}$, we have
\[\|\mu\|\le \sum_{n=1}^{\infty}|\aa_n| \|\delta_{x_n}\|=\|\mu\|_1.\]
On the other hand, if for $\mu=\sum_{n=1}^{\infty}\aa_n \delta_{x_n}$, we denote by $I=\{n: \aa_n\ge 0\}$ then by considering 
\[g(x_n):=\begin{cases}
   d(x_n,0) &\text{ if }\ n\in I; \\
   -d(x_n,0) &\text{ otherwise},
\end{cases}\]
we get \[\langle \mu, g\rangle= \sum_{n=1}^{\infty}\aa_n g(x_n)=\sum_{n=1}^{\infty}|\aa_n|d(x_n,0)=\|\mu\|_1.\]
Note that $g\in \text{Lip}_0(M)$ and that $\lip(g)\leq 2K$ since by assumption, for every $x\ne y$,
\[|g(x)-g(y)|\le d(x,0)+d(y,0)\leq  Kd(x,y)+Kd(x,y)=2Kd(x,y).\]
It follows that
\[\|\mu\|\geq \frac{\langle \mu, g\rangle}{\text{Lip}(g)}\geq \frac{1}{2K}\|\mu\|_1,\]
which finishes the proof.
\end{proof}

We observe that metric spaces $M_V$ associated with countable directed trees $V$ through the construction in Subsection \ref{subsec:shifts-M} (Proposition \ref{prop:backward-Tf}) produce such Lipschitz-free spaces since $M_V$ is indeed countable and radially discrete with respect to $0$, as for all $v\in V\backslash\{0\}$ and all $w\in V\backslash\{v\}$, we have
        \[d(v,w)=\mu_v+\mu_w\ge \mu_v= d(v,0).\]

\begin{remark}\label{rem:deltas-retract}
    A Schauder basis made by deltas is a particular case of a \textit{retractional Schauder basis}---see \cite{Nov20}---, that is, a Schauder basis on a Lipschitz-free space where the associated canonical projections to the first $n$ coordinates $(P_n)$ are all given by linearizations of Lipschitz retractions in the metric space. In this case, $P_n=T_{f_{M_n}}$, where $f_{M_n}$ is the Lipschitz retraction defined in Equation \eqref{eq:retrac}.
\end{remark}

\begin{example}\label{deltaexamp}
If $M$ is the metric space defined in \cite[Proposition 1.6]{ACP21} or if $M$ is given by $\{q^n:n\in \mathbb{Z}\}\cup \{0\}$ with $q\in \mathbb{R}\backslash\{0\}$ endowed with 
the absolute value then $\mathcal{F}(M)$ admits a Schauder basis given by 
deltas. Actually, if $M$ is a subset of $\mathbb{R}$ endowed with the absolute value then $\free{M}$ admits a 
Schauder basis given by deltas if and only if there exists $\varepsilon>0$ such that for all $x\in M\backslash\{0\}$, all $y\in M\backslash\{x\}$, we have
\[\frac{y}{x}\notin \left]1-\varepsilon,1+\varepsilon\right[.\]
 In particular, if $M=\{\frac{1}{n}\}_{n=1}^\infty \cup\{0\}$, then 
$\mathcal{F}(M)$ does not admit a Schauder basis given by deltas. However, it is well-known that $\free{M}$ is isometrically isomorphic to $\ell^1(\NN)$, through the identification $\Phi\colon \free{M}\to \ell^1(\NN)$ given by $\Phi(m_{\frac{1}{n+1},\frac{1}{n}})=e_n$, so the sequence $\{m_{\frac{1}{n+1},\frac{1}{n}}\}_{n=1}^\infty$ is a Schauder basis---which, moreover, is easily seen to be a retractional one.
\end{example}

We will finish this subsection proving that, when the Lipschitz-free space has a Schauder basis made by deltas, targeting property characterizes hypercyclicity (and weakly mixingness) of $T_f$ operators.

\begin{theorem}\label{hypdeltas}
Let $(M,d)$ be a metric space with a distinguished point $0$ and $f\in \text{Lip}_0(M,M)$. If $M$ is countable and radially discrete with respect to $0$, then the following assertions are equivalent:
\begin{enumerate}
\item $T_f$ is hypercyclic on $\free{M}$;
\item $T_f$ is weakly mixing on $\free{M}$;
\item $f$ has the \aim{} property around $0\in M$.
\end{enumerate}
\end{theorem}
\begin{proof}
Let $(x_n)_{n\geq1}$ be an enumeration of $M\backslash\{0\}$. We know by Corollary \ref{cor:delta-basis} that $(\delta_{x_n})_{n\geq1}$ is a Schauder basis in $\free{M}$.
Moreover, we have (3) $\Rightarrow$ (2) $\Rightarrow$ (1) by Theorem~\ref{thm:aim-WHC}.\\

(1) $\Rightarrow$ (3). Assume that $T_f$ is hypercyclic and thus topologically transitive.\\

Let $\varepsilon>0$, $k\ge 1$ and $m_0\ge 1$. We want to show that there exists $m\ge m_0$ such that for every $1\le j\le k$, 
\begin{itemize}
\item there exists $x\in M$ such that $d(x,0)<\varepsilon$  and $d(f^m(x),f^m(x_j))<\varepsilon$,
\item there exists $x'\in M$ such that $d(x',0)<\varepsilon$ and $f^m(x')=x_j$.
\end{itemize}
This will imply that $f$ has the targeting property around $0$.\\

To this end, we let $M_k=\{x_1,\dots,x_k\}$ and
\[\mu_k=\sum_{j=1}^k 3^j\delta_{x_j} \quad \text{and}\quad \nu_k=\sum_{j=1}^k 3^{k+j}\delta_{x_j}.\]
Therefore, by topological transitivity of $T_f$, there exist an integer $m\ge m_0$ and $\rho\in \mathcal{F}(M)$ such that 
\[\|\rho-\mu_k\|_1<\gamma \quad \text{and}\quad \|T_f^m\rho-\nu_k\|_1<\gamma\]
where $\gamma$ is given by
\[\gamma=\min\left\{\frac{\min\{d(x_j,0):1\le j\le k\}}{k},\varepsilon\right\}.\]
By density of $\text{vect}\{\delta_{x_n}: n\ge 1\}$, we can assume that $\rho=\sum_{j=1}^N a_j \delta_{x_j}$ with $N\ge k$. We then have
\[\left\|\sum_{j=1}^N a_j \delta_{x_j}-\sum_{j=1}^k 3^j\delta_{x_j}\right\|_1<\gamma \quad \text{and}\quad  \left\|\sum_{j=1}^N a_j \delta_{f^m(x_j)}-\sum_{j=1}^k 3^{k+j}\delta_{x_j}\right\|_1<\gamma.\]

In particular, we deduce by using the value of $\gamma$ that
\begin{enumerate}
\item for every $1\le j\le k$, $\displaystyle |a_j-3^j|<\frac{1}{k}$,
\item  $\displaystyle \sum_{k< j\le N}|a_j|d(x_j,0)<\varepsilon$,
\item for every $1\le j\le k$, $\displaystyle\left|\sum_{1\le i\le N:f^m(x_i)=x_j}a_i-3^{k+j}\right|<1$,
\item for every $x\in M\backslash M_k$,
\[\displaystyle\left|\sum_{1\le i\le N:f^m(x_i)=x}a_i\right| d(x,0)<\varepsilon.\]
\end{enumerate}
It follows from $(1)$ that for every set $F\subset \{1,\dots,k\}$ and every $1\le j\le k$, we have
\[\left|\sum_{i\in F} a_i-3^{k+j}\right|> 2.\]

Indeed, by (1), we have
\begin{align*}
     \left|\sum_{i\in F}a_i\right| < \sum_{i=1}^k\left(3^i+\frac{1}{k}\right)= \dfrac{3^{k+1}-1}{2},
\end{align*}
and from this follows
\begin{align*}
    \left|\sum_{i\in F} a_i-3^{k+j}\right | \geq 3^{k+j}-\dfrac{3^{k+1}-1}{2} >\frac{3^{k+1}}{2}> 2.
\end{align*}
We then deduce from $(1)$ and $(3)$ that for every $1\le j\le k$, 
\[\sum_{i\in ]k,N]:f^m(x_i)=x_j}|a_i|\ge \left|\sum_{i\in ]k,N]:f^m(x_i)=x_j}a_i\right|>1\]

since
\begin{align*}
\left|\sum_{i\in ]k,N]:f^m(x_i)=x_j}a_i\right|&=\left|\left(\sum_{1\le i\le N:f^m(x_i)=x_j}a_i-3^{k+j}\right)-\left(\sum_{1\le i\le k:f^m(x_i)=x_j}a_i-3^{k+j}\right) \right|\\
&\ge \left|\left(\sum_{1\le i\le k:f^m(x_i)=x_j}a_i-3^{k+j}\right) \right|-\left|\left(\sum_{1\le i\le N:f^m(x_i)=x_j}a_i-3^{k+j}\right)\right|\\
&>2-1=1.
\end{align*}

We can now conclude from $(2)$ that for every $1\le j\le k$, there exists $i\in ]k,N]$ such that $f^m(x_i)=x_j$ and $d(x_i,0)<\varepsilon$.\\

It remains to show that for every $1\le j\le k$, there exists $x\in M$ such that $d(x,0)<\varepsilon$  and $d(f^m(x),f^m(x_j))<\varepsilon$.  
Let $1\le j\le k$. 
If $d(f^m(x_j),0)<\varepsilon$ then it suffices to consider $x=0$. From now, we will thus assume that $d(f^m(x_j),0)\ge \varepsilon$ and we remark that it follows from $(1)$ that for every non-empty set $F\subset \{1,\dots,k\}$, we have
\begin{equation}\label{eq_coeff}
    \left|\sum_{i\in F}a_i\right|> 2.
\end{equation}

There are now two possibilities:
\begin{itemize}
\item If $f^m(x_j)\notin M_k$, then we deduce from $(4)$ that 
\[\left|\sum_{1\le i\le N:f^m(x_i)=f^m(x_j)}a_i\right| d(f^m(x_j),0)<\varepsilon\]
and since we have assumed that $d(f^m(x_j),0)\ge \varepsilon$. We deduce that 
\[\left|\sum_{1\le i\le N:f^m(x_i)=f^m(x_j)}a_i\right|<1.\]
Combined with \eqref{eq_coeff} for $F=\{1\le i\le k:f^m(x_i)=f^m(x_j)  \}\ni j$, it follows that  
\[\left|\sum_{i\in ]k,N]:f^m(x_i)=f^m(x_j)}a_i\right|\ge 1.\]
Therefore, by $(2)$, there exists $i\in ]k,N]$ such that $f^m(x_i)=f^m(x_j)$ and $d(x_i,0)<\varepsilon$. Thus we can consider $x=x_i$ in this case.

\item If $f^m(x_j)\in M_k$, then $f^m(x_j)=x_J$ for some $1\le J\le k$ and we use $(3)$ to get that 
\[\left|\sum_{1\leq i\leq N:f^m(x_i)=f^m(x_j)}a_i-3^{k+J}\right|<1.\]
We can then deduce from $(1)$ that
\[\left|\sum_{i\in ]k,N]:f^m(x_i)=f^m(x_j)}a_i\right|\ge 1\]
thus by $(2)$, there exists $i\in ]k,N]$ such that $f^m(x_i)=f^m(x_j)$ and $d(x_i,0)<\varepsilon$. Then, we can consider $x=x_i$ in this case.
\end{itemize}

\end{proof}

\begin{remark}
    An easy modification of the statement above would characterize the mixing property of $T_f$ if and only if $f$ has the strict targeting property.
\end{remark}

\begin{remark} As mentioned in Section~\ref{subsec:shifts-M} in the case of $M_V$, any operator $T_f$ satisfying the conditions of Theorem \ref{hypdeltas} can be understood as a slight generalization of a backward shift on trees. Indeed, the space $M$ on the statement will satisfy that $\free{M}$ is isomorphic to $\ell^1(M\backslash\{0\},\mu)$ for $\mu_x=d(x,0)$. This is done thanks to the isomorphism generated by $\Phi(e_n):= \delta_{x_n}$. In this case, 
    \[T_f(\delta_{x_n})= \delta_{f(x_n)}=\Phi(e_{f(x_n)}).\]
    Through this identification, $T_f$ would be conjugated to the operator $B_{f}$ defined on $\ell^1(M\backslash\{0\},\mu)$ by 
    \[B_{f}(e_{x_n}):= e_{f(x_n)}.\]
    
    Notice that such an operator can be hypercyclic and might present some differences with the formal definition of backward shifts on trees, since the map $f$ is allowed to have \textit{unconnected components} (two different components $A,B\subset M$ such that no orbit of an element of $A$ has intersection with an orbit of an element of $B$), or to have \textit{cycles} (some orbits can be eventually periodic).
\end{remark}

As mentioned in Example~\ref{deltaexamp}, if we consider the metric space $M=\{\frac{1}{n}\}_{n=1}^\infty \cup\{0\}$, then 
$\mathcal{F}(M)$ does not admit a Schauder basis given by deltas. One can wonder if we can therefore find a map $f\in \lip_0(M,M)$ such that $T_f$ is hypercyclic and $f$ has not the targeting property around $0$. However, even if $\mathcal{F}(M)$ does not admit a Schauder basis of deltas, we are going to show that for this space, we keep the equivalence between $f$ has the targeting property around $0$ and $T_f$ is hypercyclic.

\subsection{When points far from $0$ are well separated}

In this subsection, we will extend the previous technique by considering a more general family of metric spaces.

\begin{definition}
    Let $(M,d)$ be a metric space, $x_0\in M$. We say that $M$ is \textbf{asymptotically uniformly discrete} (with respect to $x_0$) if for every $\varepsilon>0$, there exists $r_{\varepsilon}>0$ such that for all $x\ne y$,
\[d(x,y)<r_{\varepsilon}\quad \implies x,y\in B(x_0,\varepsilon).\]
\end{definition}

In particular, the metric spaces considered in Theorem~\ref{hypdeltas} satisfy this condition but also the metric space $M=\{1/n:n\ge 1\}\cup\{0\}$ endowed with the absolute value. This condition on the separation of points far from $0$ will allow us to deduce convenient inequalities on the coefficients $(a_i)_{1\le i\le k}$ and $(b_j)_{1\le j\le N}$ from inequalities
\[\left\|\sum_{i=1}^k a_i \delta_{x_i}-\sum_{j=1}^N b_j \delta_{y_j}\right\|<\gamma\]
when $\gamma$ is sufficiently small. These inequalities are the key point in the proof and thus in the generalization of Theorem~\ref{hypdeltas} to this bigger family of metric spaces.

\begin{proposition}\label{gamma}
Let $M$ be a metric space with a distinguished point $0\in M$. If $M$ is asymptotically uniformly discrete, then for every sum $\sum_{i=1}^k a_i \delta_{x_i}$ with $a_i\ne 0$ and $x_i\ne 0$, for every $\varepsilon>0$, there exists $\gamma>0$ such that if 
\[\left\|\sum_{i=1}^k a_i \delta_{x_i}-\sum_{j=1}^N b_j \delta_{y_j}\right\|<\gamma\]
with $(y_j)_{1\le j\le N}$ pairwise disjoint, then there exists an injective map $\varphi\colon [1,k]\to [1,N]$ such that
\begin{enumerate}
\item for every $1\le i\le k$, 
\[x_i=y_{\varphi(i)} \quad \text{and}\quad |a_i-b_{\varphi(i)}|<\varepsilon\]
\item we have 
\[\sum_{j\in [1,N]\backslash \varphi([1,k]):y_j\notin B(0,\varepsilon)}|b_j|<1.\]
\end{enumerate}
\end{proposition}
\begin{proof}
Let $\nu=\sum_{i=1}^k a_i \delta_{x_i}$ with $a_i\ne 0$ and $x_i\ne 0$. Let $\varepsilon>0$. Since each nonzero point in $M$ is isolated, if for each finite subset $F$ of $M\backslash\{0\}$ we let 
\[f_{F}(x):=\begin{cases}
      1 &\text{ if}\ x\in F \\
  0 &\text{ otherwise },
\end{cases}\]
then $f_F\in \text{Lip}_0(M)$ because 
\[\lip(f_{F})\le \max_{x\in F} \frac{1}{\inf_{y\ne x} d(x,y)}=\max_{x\in F}\lip(f_{\{x\}})<\infty.\]

Then, we let
\[L=\max_{x\in \{x_1,\dots,x_k\}}\lip(f_{\{x\}})\quad  \text{and} \quad  \gamma=\frac{1}{2}\min\left\{\frac{\min\{|a_i|:1\le i\le k\}}{L},\frac{\varepsilon}{L},r_{\varepsilon} \right\}.\] 
Assume that
\[\left\|\sum_{i=1}^k a_i \delta_{x_i}-\sum_{j=1}^N b_j \delta_{y_j}\right\|<\gamma\]
with $(y_j)_{1\le j\le N}$ pairwise disjoint.\\
We first deduce that for every finite subset $F$ of $\{x_1,\dots,x_k\}$, we have
\[
\gamma>\left\|\sum_{i=1}^k a_i \delta_{x_i}-\sum_{j=1}^N b_j \delta_{y_j}\right\|\ge \frac{1}{L} \left\vert\sum_{i=1}^k a_i \delta_{x_i}(f_F)-\sum_{j=1}^N b_j \delta_{y_j}(f_F)\right\vert= \frac{1}{L} \left\vert\sum_{i=1}^k a_i f_F(x_i)-\sum_{j=1}^N b_j f_F(y_j)\right\vert.
\]
In particular, for $F=\{x_i\}$, it means that that there exists $1\le j\le N$ such that $y_j=x_i$ because if it was not the case we would have \[\gamma>\frac{|a_i|}{L},\]
contradicting the definition of $\gamma$. Moreover, $(y_j)_{1\le j\le N}$ being pairwise disjoint, for every $1\le i\le k$, there exists a unique $1\le j\le N$, that we will denote by $\varphi(i)$, such that $y_{\varphi(i)}=x_i$ and in addition,  
\[\frac{|a_i-b_{\varphi(i)}|}{L}<\gamma\]
and thus $|a_i-b_{\varphi(i)}|<\varepsilon$. Finally, if we let $J=\{j\in [1,N]\backslash \varphi([1,k]); y_j\notin  B(0,\varepsilon)\}$, and
\[f(x)=\begin{cases}
      1 &\text{ if}\ x=y_j \ \text{with $j\in J$ and $b_j\ge 0$} \\
      -1 &\text{ if}\ x=y_j \ \text{with $j\in J$ and $b_j< 0$} \\      
  0 &\text{ otherwise }
\end{cases},\]
then $f\in \text{Lip}_0(M)$ because by assumption,
\begin{align*}
\lip(f)=\sup_{x\ne y}\frac{|f(x)-f(y)|}{d(x,y)}
&\le \max_{j\in J}\sup_{x\ne y_j}\frac{|f(x)-f(y_j)|}{d(x,y_j)}\\
&\le 2 \max_{j\in J}\sup_{x\ne y_j}\frac{1}{d(x,y_j)}\\
&\le 2 \frac{1}{r_{\varepsilon}}\le \frac{1}{\gamma}.
\end{align*}

Therefore we can conclude that
\begin{align*}
\sum_{j\in J}|b_j|
=\left|-\sum_{j\in J}b_jf(y_j)\right|
&=\left|\sum_{i=1}^k a_if(x_i)-\sum_{j=1}^{N}b_jf(y_j)\right|\\
&=\left|\sum_{i=1}^k a_i\delta_{x_i}(f)-\sum_{j=1}^{N}b_j\delta_{y_j}(f)\right|\\
&\le \frac{1}{\gamma}\left\|\sum_{i=1}^k a_i\delta_{x_i}-\sum_{j=1}^{N}b_j \delta_{y_j}\right\|< 1.
\end{align*}
\end{proof}

\begin{theorem}\label{hypdeltas2}
Let $(M,d)$ be a metric space with a distinguished point $0$ and $f\in \text{Lip}_0(M,M)$. If $M$ is countable and asymptotically uniformly discrete, then the following assertions are equivalent:
\begin{enumerate}
\item $T_f$ is hypercyclic on $\free{M}$;
\item $T_f$ is weakly mixing on $\free{M}$;
\item $f$ has the \aim{} property around $0\in M$.
\end{enumerate}
\end{theorem}
\begin{proof}
We know that  (3) $\Rightarrow$ (2) $\Rightarrow$ (1). It remains to show that (1) $\Rightarrow$ (3) by adapting the proof of Theorem~\ref{hypdeltas}.\\

Assume that $T_f$ is hypercyclic and thus topologically transitive. Let $(x_n)_{n\ge 1}$ be an enumeration of $M\backslash\{0\}$.
Let $k\ge 1$, $0<\varepsilon<1/k$ and $m_0\ge 1$. We want to show that there exists $m\ge m_0$ such that for every $1\le j\le k$, 
\begin{itemize}
\item there exists $x\in M$ such that $d(x,0)<\varepsilon$  and $d(f^m(x),f^m(x_j))<\varepsilon$,
\item there exists $x'\in M$ such that $d(x',0)<\varepsilon$ and $f^m(x')=x_j$.
\end{itemize}
This will imply that $f$ has the targeting property around $0$.\\

To this end, let us define $M_k=\{x_1,\dots,x_k\}$ and
\[\mu_k=\sum_{j=1}^k 3^j\delta_{x_j} \quad \text{and}\quad \nu_k=\sum_{j=1}^k 3^{k+j}\delta_{x_j}.\]
Therefore, by topological transitivity of $T_f$, there exist $m\ge m_0$ and $\rho\in \mathcal{F}(M)$ such that 
\[\|\rho-\mu_k\|<\gamma \quad \text{and}\quad \|T_f^m\rho-\nu_k\|<\gamma\]
where $\gamma$ is a positive real number satisfying
\[\gamma\le\min\left\{\frac{\min\{d(x_j,0):1\le j\le k\}}{k},\varepsilon\right\}\]
and sufficiently small to get the conclusion of Proposition~\ref{gamma} for $\mu_k$ and $\varepsilon$ and also for $\nu_k$ and $\varepsilon$. 
By density of $\sp\{\delta_{x_n}: n\ge 1\}$, we can assume that $\rho=\sum_{j=1}^N a_j  \delta_{x_j}$ with $N\ge k$. We then have
\[\left\|\sum_{j=1}^N a_j \delta_{x_j}-\mu_k\right\|<\gamma \quad \text{and}\quad \left\|\sum_{j=1}^N a_j \delta_{f^m(x_j)}-\nu_k\right\|<\gamma.\]

In particular, we deduce by using the value of $\gamma$ and   Proposition~\ref{gamma} that
\begin{enumerate}
\item for every $1\le j\le k$, $\displaystyle |a_j-3^j|<\varepsilon<1/k$, \label{cond:I}
\item  $\displaystyle \sum_{k< i\le N: x_i\notin B(0,\varepsilon)}|a_i|<1$,\label{cond:II}
\item for every $1\le j\le k$, $\displaystyle\left|\sum_{1\le i\le N:f^m(x_i)=x_j}a_i-3^{k+j}\right|<\varepsilon<1$,\label{cond:III}
\item for every $x\in M\backslash M_k$ with $x\notin B(0,\varepsilon)$
\[\displaystyle\left|\sum_{1\le i\le N:f^m(x_i)=x}a_i\right|<1.\]\label{cond:IV}
\end{enumerate}
As in the proof of Theorem~\ref{hypdeltas}, it follows from \eqref{cond:I} that for every set $F\subset \{1,\dots,k\}$ and every $1\le j\le k$, we have
\[\left|\sum_{i\in F} a_i-3^{k+j}\right|> 2.\]
Thus, we can deduce from \eqref{cond:I}  and \eqref{cond:III} that for every $1\le j\le k$, 
\[\sum_{i\in ]k,N]:f^m(x_i)=x_j}|a_i|\ge \left|\sum_{i\in ]k,N]:f^m(x_i)=x_j}a_i\right|>1.\]
Hence, we can now conclude from \eqref{cond:II} that for every $1\le j\le k$, there exists $i\in ]k,N]$ such that $f^m(x_i)=x_j$ and $d(x_i,0)<\varepsilon$.\\

It remains to show that for every $1\le j\le k$, there exists $x\in M$ such that $d(x,0)<\varepsilon$  and $d(f^m(x),f^m(x_j))<\varepsilon$.  
Let $1\le j\le k$. 
Again if $d(f^m(x_j),0)<\varepsilon$ then it suffices to consider $x=0$ and we can thus assume from now that $d(f^m(x_j),0)\ge \varepsilon$, i.e. $f^m(x_j)\notin B(0,\varepsilon)$. In particular, it will allow us to use (4) with $x=f^m(x_j)$ when $f^m(x_j)\notin M_k$. We can therefore finish the proof similarly to the proof of Theorem~\ref{hypdeltas} by first remarking that thanks to \eqref{cond:I}, for every non-empty set $F\subset \{1,\dots,k\}$, we have
\begin{equation}\label{eq_coeff2}
    \left|\sum_{i\in F}a_i\right|> 2
\end{equation}
and by considering both possibilities:
\begin{itemize}
\item If $f^m(x_j)\notin M_k$, then we deduce from \eqref{cond:IV} that 
\[\left|\sum_{1\le i\le N:f^m(x_i)=f^m(x_j)}a_i\right|<1.\]
From \eqref{eq_coeff2}, we deduce that
\[\left|\sum_{i\in ]k,N]:f^m(x_i)=f^m(x_j)}a_i\right|\ge 1.\]
Thus, by \eqref{cond:II}, there exists $i\in ]k,N]$ such that $f^m(x_i)=f^m(x_j)$ and $d(x_i,0)<\varepsilon$. Finally, we consider $x=x_i$ in this case.

\item If $f^m(x_j)\in M_k$, then $f^m(x_j)=x_J$ for some $1\le J\le k$ and we use \eqref{cond:III} to get that 
\[\left|\sum_{1\le i\le N:f^m(x_i)=f^m(x_j)}a_i-3^{k+J}\right|<1.\]
We can then deduce from \eqref{cond:I} that
\[\left|\sum_{i\in ]k,N]:f^m(x_i)=f^m(x_j)}a_i\right|\ge 1.\]
Hence, by \eqref{cond:II}, there thus exists $i\in ]k,N]$ such that $f^m(x_i)=f^m(x_j)$ and $d(x_i,0)<\varepsilon$. We consider $x=x_i$ in this case.
\end{itemize}
\end{proof}

With this characterization, we are able to provide a strengthening of \cite[Remark 1.3]{ACP21}.

\begin{corollary}\label{cor:ud-no-HC}
Let $(M,d)$ be a non-trivial uniformly discrete metric space, i.e. $M$ is not a singleton and there exists $\varepsilon>0$ such that $d(x,y) > \varepsilon$ for
every $x\ne y$. Then, there is no $f\in\lip_0(M,M)$ such that $T_f$ is hypercyclic, for any choice of $0\in M$.

\end{corollary}
\begin{proof}
Assume by contradiction that there exists $0\in M$ and $f\in \lip_0(M,M)$ such that $T_f$ is hypercyclic. Notice that any uniformly discrete metric space is in particular asymptotically uniformly discrete and thus, by Theorem \ref{hypdeltas2}, $f\in\lip_0(M,M)$ has the targeting property around $0\in M$. However, notice that if a map has the targeting property around a fixed point $x_0$ then it is necessary for that point $x_0$ to be not isolated (or $M=\{x_0\}$). This contradicts the fact the $M$ is uniformly discrete.
\end{proof}

We conclude this section by giving an equivalent formulation of the targeting property on the particular class of metric spaces considered here.

\begin{proposition}
Let $(M,d)$ be an asymptotically uniformly discrete metric space with a distinguished point $0$ and $f\in \text{Lip}_0(M,M)$. Then $f$ has the targeting property around $0$ if and only if there exists an increasing sequence $(n_k)_{k\geq1}$ such that for any $x\in M$,
\[ d(f^{-n_k}(\{x\}),0) \xrightarrow[k\to +\infty]{} 0
\quad \text{and }\quad \min\{d(f^{n_k}(x),0), d(f^{-n_k}(\{f^{n_k}(x)\}),0)\}\xrightarrow[k\to +\infty]{} 0.\]
\end{proposition}
\begin{proof}
We can already remark that each non-zero point is isolated and therefore, if $N$ is dense in $M$ then $N=M\backslash \{0\}$ or $M$. Moreover, for $x=0$, if we consider $x_k=y_k=z_k=0$, then $\lim_{k\to\infty} x_k=0$,
$\lim_{k\to\infty} z_k=0$, $\lim_{k\to\infty} d(f^{n_k}(z_k),f^{n_k}(x_k))=0$ and  $\lim_{k\to\infty} f^{n_k}(y_k)=0$. In other words, $f$ has the targeting property around $0$ if and only if
\begin{enumerate}
\item[(i)] For each $x\in M\backslash \{0\}$, there exists two sequences $(x_k)_{k\geq1}\in M^{\N}$ and $(z_k)_{k\geq1}\in M^{\N}$ such that $\lim_{k\to\infty} x_k=x$, $\lim_{k\to\infty} z_k=0$ and $\lim_{k\to\infty} d(f^{n_k}(z_k),f^{n_k}(x_k))=0$, 
\item[(ii)] For each $y\in M\backslash \{0\}$, there exists a sequence $(y_k)_{k\geq1}\in M^\N$ such that $\lim_{k\to\infty} y_k=0$ and $\lim_{k\to\infty} f^{n_k}(y_k)=y$. 
\end{enumerate}
Now, since each non-zero point is isolated, we deduce that for every $x\in M\backslash \{0\}$, we have ultimately $x_k=x$ and for every $y\in M\backslash \{0\}$, we have ultimately $f^{n_k}(y_k)=y$.
Therefore, if $f$ has the targeting property around $0$ then 
$\lim_k d(f^{-n_k}(\{y\}),0)=0$ and there exists $(z_k)$ such that for every $\varepsilon>0$, there exists $k_0$ such that for every $k\ge k_0$, we have
\[d(z_k,0)<\varepsilon \quad \text{and}\quad d(f^{n_k}(z_k),f^{n_k}(x))<r_{\varepsilon}.\]
The second inequality implies that either $f^{n_k}(x)=f^{n_k}(z_k)$ or $f^{n_k}(x)\in B(0,\varepsilon)$. In both cases, we get 
\[\min\{d(f^{n_k}(x),0), d(f^{-n_k}(\{f^{n_k}(x)\}),0)\}<\varepsilon.\]
\text{}\\
On the other hand, if there exists an increasing sequence $(n_k)$ such that for any $x\in M$,
\[\lim_{k\to\infty} d(f^{-n_k}(\{x\}),0)=0
\quad \text{and }\quad \lim_{k\to\infty} \min\{d(f^{n_k}(x),0), d(f^{-n_k}\{f^{n_k}(x)\},0)\}=0,\]
then we get the targeting property for $f$ around $0$:
\begin{enumerate}
\item[(i)] for each $x\in M\backslash \{0\}$, if we let $d_k:=\min\{d(f^{n_k}(x),0), d(f^{-n_k}(\{f^{n_k}(x)\}),0)\}$, we can pick $x_k:=x$ and 
\begin{equation*}
    z_k:= \begin{cases}
       0 & \text{ if } d_k=d(f^{n_k}(x),0); \\
       z_k\in f^{-n_k}(\{f^{n_k}(x)\}) \cap  B(0,2d_k) & \text{ otherwise.}
    \end{cases}
\end{equation*}
 As $\lim_{k\to\infty} d_k=0$, we deduce that $z_k\xrightarrow{k\to \infty}0$. Also,
 \[d(f^{n_k}(z_k),f^{n_k}(x_k))= d(f^{n_k}(z_k),f^{n_k}(x)) = 
 \begin{cases}
       d_k & \text{ if } z_k=0; \\
      0 & \text{ otherwise}.
 \end{cases}\]
\item[(ii)] for each $y\in M\backslash \{0\}$, as $\lim_{k\to\infty} d(f^{-n_k}(\{y\}),0)=0$, we can pick $y_k\in f^{-n_k}(\{y\})$ such that $y_k\xrightarrow[]{k\to\infty} 0$.
\end{enumerate}
This finishes the proof.
\end{proof}

\section{Conclusions}

We observed that Lipschitz-free operators constitute a universal model among linearizations of Lipschitz maps with a fixed point on metric spaces. 
We introduced new criteria to ensure the (weak) mixing property for a linear operator by achieving concrete conditions on a (not necessarily linear) restriction of it. As a byproduct, we have been able to extend the study of Lipschitz-free operators within the dynamical point of view, obtaining significant improvements on previous results within the literature. In particular, we obtained that, for the introduced class of asymptotically uniformly metric spaces, the weak mixing property and hypercyclicity of $T_f$ can be characterized by the targeting property of $f$ (Theorem \ref{hypdeltas2}). We recall (see \cite{ACP21}) that it is not  known if there exist Lipschitz-free operators being hypercyclic and not weakly mixing. 
The results in \cite{ACP21}, and our Theorem \ref{hypdeltas2}, point to a negative answer. Moreover, in the case that the underlying metric space $M$ is bounded, it is readily seen that for any $T_f$ operator, the set $\sp(\dd(M))$ will be a dense set of points with bounded orbit; then, as a straightforward application of \cite[Theorem 2.48]{GEP11} any hypercyclic $T_f$ operator on a Lipschitz-free space of a bounded metric space  $M$ will be weakly mixing\footnote{The first author thanks Colin Petitjean for noticing him about this observation.}. 
We also recall that finding an hypercyclic operator that is not weakly-mixing is not an easy task \cite{RosaRead}. Another property extremely difficult to achieve is the absence of non-trivial invariant closed subsets. It is well-known that such operators exist on $\ell^1$ \cite{Read}. However, we notice that every $T_f$ always contains $\dd(M)$ as a non-trivial closed invariant set.\\

The targeting property  that we have introduced in this paper allows us to characterize hypercyclic Lipschitz-free operators on $\free{M}$ for different kinds of countable metric spaces $M$ where each non-zero point is isolated. Our results do not cover all these spaces, although they provide some hope for a characterization of hypercyclic Lipschitz-free operators $T_f$ in terms of the map $f$ in this context. 

\begin{problem}
Let $(M,d)$ be a countable metric space with a distinguished point $0$ where each non-zero point is isolated. How to characterize the maps $f\in \lip_0(M,M)$ for which the Lipschitz-free operator $T_f$ is hypercyclic? 
\end{problem}

\section*{AI Disclosure Statement}
The authors declare that no generative AI models were used at any stage in the elaboration of this manuscript.

\section*{Acknowledgements}
This work was partially developed while the first author did research stays at Université de Mons (2024) and Université de Lille (2025). The first author gratefully acknowledges the warm hospitality at both institutions as well as their departments and research teams during his visit. The first author was partially supported by Generalitat Valenciana (through the predoctoral fellowship CIACIF/2021/378, and the research stay grant CIBEFP/2023/20), by MICIU/AEI/10.13039/
501100011033/FEDER, UE (through Projects PID2021-122126NB-C33 and  PID2022-139449NB-I00), and by the Fundació Ferran Sunyer i Balaguer (through Borsa de Viatge Ferran Sunyer i Balaguer 2025).

The second author was partially supported by the French ANR project No. ANR-20-CE40-0006.

The third author was partially supported by the F.R.S-FNRS through funding PDR No. T.0285.26.

Part of the research was done while the fourth author was invited by Université du Littoral Côte d’Opale. He acknowledges the hospitality and support. The fourth author was partially supported by MICIU/AEI/10.13039/501100011033/FEDER, UE, Project  PID2022-139449NB-I00.

\printbibliography

\end{document}